\documentclass[11pt,reqno]{amsart}
\usepackage[top=30truemm,bottom=30truemm,left=30truemm,right=30truemm]{geometry} 
\usepackage{amsmath, amssymb, array}
\usepackage[all]{xy}
\usepackage{ascmac}
\usepackage[dvips]{graphicx}
\usepackage{color}
\usepackage{amscd}
\usepackage[driverfallback=dvipdfm]{hyperref}
\usepackage{amsrefs}
\usepackage{cleveref}
\usepackage{shuffle}
\usepackage{mathrsfs}
\usepackage{listings} 
\usepackage{comment}
\usepackage{multirow,bigdelim}

\usepackage{mathtools}

\newfont{\cyr}{wncyr10 scaled\magstep0}

\newcommand{\C}{\mathbb{C}}

\newcommand{\F}{\mathbb{F}}

\newcommand{\HH}{\mathbb{H}}

\newcommand{\Q}{\mathbb{Q}}
\newcommand{\R}{\mathbb{R}}

\newcommand{\Z}{\mathbb{Z}}

\DeclareMathOperator{\Gal}{Gal} 
\DeclareMathOperator{\Tr}{Tr} 

\theoremstyle{plain}
 \newtheorem{theorem}{Theorem}[section]
 \crefname{theorem}{Theorem}{Theorems}
 \newtheorem{proposition}[theorem]{Proposition}
 \crefname{proposition}{Proposition}{Propositions}
 \newtheorem{lemma}[theorem]{Lemma}
 \crefname{lemma}{Lemma}{Lemmas}
 \newtheorem{corollary}[theorem]{Corollary}
 \crefname{corollary}{Corollary}{Corollaries}
 
 \crefname{conjecture}{Conjecture}{Conjectures}
 
 \crefname{hypothesis}{Hypothesis}{Hypotheses}
 
 \crefname{question}{Question}{Questions}
 
 \crefname{problem}{Problem}{Problems}

\theoremstyle{definition} 
 \newtheorem{definition}[theorem]{Definition}
 \crefname{definition}{Definition}{Definitions}
 
 \crefname{example}{Example}{Examples}
 \newtheorem{remark}[theorem]{Remark}
 \crefname{remark}{Remark}{Remarks}

\title{Extending  van der Geer and Hirzebruch's generalization of Construction A}
\author{Riku Higa}
\address[Riku Higa]{Department of Mathematics \\ Faculty of Science and Technology \\ Tokyo University of Science, 2641, Yamazaki, Noda, Chiba, Japan}
\email{6123702@ed.tus.ac.jp \\ 6121510@ed.tus.ac.jp}

	\subjclass[2010]{primary 11H71; 
	secondary
		11E12; 
		11H06; 
	}
\keywords{unimodular lattice; coding theory; theta series.}

\begin{document}


\maketitle

\begin{abstract}
This paper discusses a generalization of Construction A and its application to root lattices of type $A_n, D_n,$  and $E_n$. While the classical Construction A provides a correspondence between binary codes and lattices, and its extension to cyclotomic fields relates $p$-ary codes to lattices over the ring of integers of $\mathbb{Q}(\zeta_p)$, we further extend this framework to a broader class of even root lattices.
Specifically, for a given lattice $\Lambda$ of type $A_n, D_n,$  or $E_n$, we define a surjective mapping $\rho$ from the dual lattice $\Lambda^*$ to a finite ring $R$ isomorphic to the quotient $\Lambda^*/\Lambda$. We then characterize the conditions under which the resulting lattice $(\rho^{\oplus m})^{-1}(C)$, constructed from an $R$-code $C$ of length $m$, becomes an even unimodular lattice.
Furthermore, we apply our generalized Construction A to the algebraic lattices of type $D_4$ and $E_6$ realized as fractional ideals of cyclotomic fields. 
Consequently, we show that the theta functions associated with these lattices can be expressed via the weight enumerators of the corresponding codes, proving that they are Hilbert modular forms over totally real fields.

\end{abstract}



\section{Introduction}

The relation between lattices and codes has been studied by a number of researchers.
The foundational example of these studies is the construction of a lattice from a binary linear code, known as Construction A.
This leads to a direct connection between the Hamming weight enumerator of a binary linear code and the theta function of the corresponding lattice (see, e.g., \cite[Ch. 7]{conway1998sphere} for details). 
van der Geer and Hirzebruch generalized this relation from the binary case to self-dual codes over the prime field $\F_p$, where $p$ is an odd prime number (see, e.g., \cite[Ch. 5]{Ebeling} for details).
One of the key features of their generalization is a correspondence between lattices of type $A_{p-1}$ and $p$-ary codes.

In this paper, we extend this correspondence to a broader class of even root lattices. Our main result is summarized in the following theorem.

\begin{theorem}\label{mainthe}
	Let $n$ and $m$ be positive integers, let $\Lambda$ be a root lattice of type $A_n$, $D_n$, or $E_n$, and let $C$ be an $R$-code of length $m$, where the finite ring $R$ is defined by
	\[
	R = \begin{cases}
		\Z/(n+1)\Z & \text{if } \Lambda \text{ is of type } A_n, \\
		\Z/4\Z      & \text{if } \Lambda \text{ is of type } D_n \text{ with } n \text{ odd}, \\
		\F_2[u]/(u^2)  & \text{if } \Lambda \text{ is of type } D_n \text{ with } n \text{ even}, \\ 
		\Z/3\Z      & \text{if } \Lambda \text{ is of type } E_6, \\
		\Z/2\Z      & \text{if } \Lambda \text{ is of type } E_7. 
	\end{cases}
	\] 
	Then, there exists a surjective mapping $\rho^{\oplus m}: (\Lambda^*)^{\oplus m} \rightarrow R^{\oplus m}$ such that the resulting lattice $\Lambda_C := (\rho^{\oplus m})^{-1}(C)$ is even unimodular if and only if the code $C$ satisfies the following conditions:
	\begin{enumerate}
		\item If $\Lambda$ is of type $A_n$ with $n$ odd, $C$ is Type II.
		\item If $\Lambda$ is of type $A_n$ with $n$ even, $C$ is Euclidean self-dual.
		\item If $\Lambda$ is of type $D_n$ with $n$ odd, $C$ is Type II.
		\item If $\Lambda$ is of type $D_n$ with $n \in 2\Z$ and $n \notin 4\Z$, $C$ is Type II.
		\item If $\Lambda$ is of type $D_n$ with $n \in 4\Z$, $C$ is Type IV.
		\item If $\Lambda$ is of type $E_6$, $C$ is Euclidean self-dual.
		\item If $\Lambda$ is of type $E_7$, $C$ is Type II.
	\end{enumerate}
\end{theorem}


\begin{remark}
	If $\Lambda$ is a root lattice of type $E_8$, it is well known that $\Lambda$ is unimodular, i.e., $\Lambda = \Lambda^*$. Consequently, $(\Lambda^*)^{\oplus m} = \Lambda^{\oplus m}$, meaning there are no non-trivial intermediate lattices between them to consider in the context of our construction.
\end{remark}

\begin{remark}\label{rem1}
	Let $n$ be an even integer and $\Lambda$ be a lattice of type $D_n$. Since the quotient group $\Lambda^*/\Lambda$ is isomorphic to $(\Z/2\Z)^{\oplus 2}$, we can choose the finite ring $R$ in Theorem \ref{mainthe} to be not only $\F_2[u]/(u^2)$ but also $\F_4$ or $\F_2\times \F_2$. We provide the corresponding results for the cases of $\F_4$ and $\F_2\times \F_2$ in \cref{Dnsub3} and \cref{Dnsub4}, respectively.
\end{remark}

\begin{remark}
	Let $\Lambda$ be a root lattice of type $A_n$, $D_n$, or $E_n$. Theorem \ref{mainthe} yields a lattice of rank $nm$ from an $R$-code of length $m$. 
	Several existing results in the literature, such as 
	\cite[Theorem 3.1]{BannaiTypeII}, 
	\cite[Theorem 12.5.12]{HuffmanFundamentals}, 
	\cite[Proposition 3.6]{article}, 
	\cite[Theorem 5]{SloaneGF(4)}, and 
	\cite[Proposition 3.6]{BachocApplications}, 
	present similar code-to-lattice constructions. However, our construction distinctively differs from theirs in terms of the rank of the resulting lattices. Specifically, for a code of length $m$, the lattices constructed in 
	\cite{BannaiTypeII} and \cite{HuffmanFundamentals} have rank $m$; 
	those in \cite{article} and \cite{SloaneGF(4)} have rank $2m$; and 
	those in \cite{BachocApplications} have rank $2m$ or $4m$.
\end{remark}

\begin{remark}
Theorem \ref{mainthe} presents a framework related to the results discussed in \cite[Chapter 18]{conway1998sphere}, yet it distinguishes itself by providing a concrete construction of the surjective mapping $\rho$ based on the specific bases of root lattices. By restricting the codes $C$ to be $R$-modules (i.e., linear codes), this formulation facilitates the systematic application to Hilbert modular forms, as demonstrated in Section \ref{Sect4}.

It is important to note, however, that while any $R$-code satisfying the prescribed conditions yields an even unimodular lattice, the converse does not necessarily hold.
Specifically, for an arbitrary even unimodular lattice $L$ satisfying the inclusions $\Lambda^{\oplus m}\subset L\subset \left(\Lambda^{*}\right)^{\oplus m}$ , its image $\rho^{\oplus m}(L)$ is always a subgroup of $R^{\oplus m}$, but it is not necessarily an $R$-module closed under scalar multiplication. As pointed out in \cite[Remark (iii), Chapter 18]{conway1998sphere} regarding the root system $3D_8$, certain self-dual subgroups corresponding to even unimodular lattices fail to be linear subspaces over the associated field or ring.
\end{remark}

As an application of our main results, we present several examples related to Hilbert modular forms.
In \cite{BayerFluckiger1999LatticesAN}, the root lattices of type $D_{4}$ and $E_{6}$ are realized as fractional ideals of $\Q(\zeta_{8})$ and $\Q(\zeta_{9})$, respectively. By combining these algebraic lattices with our construction in Theorem \ref{mainthe}, we can construct even unimodular lattices as modules over the rings of integers of these fields. The theta functions of the resulting unimodular lattices are shown to be Hilbert modular forms over the associated totally real fields. Furthermore, we establish an explicit relation between the weight enumerators of the underlying codes and these theta functions.

The organization of this paper is as follows.
In Section 2, we recall some basic definitions and notation.
In Section 3, we prove Theorem \ref{mainthe} and Remark \ref{rem1}.
In Section 4, we introduce the applications of our generalized Construction A to Hilbert modular forms, providing the explicit examples mentioned above.

\subsection*{Acknowledgement}
The author is sincerely grateful Prof.\ Hiroki Aoki for encouraging his research.

%
%

\subsection*{Notation}
For a field extension $K/k$, the usual trace and norm maps $\Tr_{K/k}: K \rightarrow k$ and $N_{K/k}: K \rightarrow k$ are often abbreviated to $\Tr$ and $N$, respectively, when the context is clear. For a field $k$ and a positive integer $n$, $k^{\oplus n}$ denotes the direct sum of $n$ copies of $k$.

We denote by $\Z, \Q, \R$, and $\C$ the ring of integers, the field of rational numbers, the field of real numbers, and the field of complex numbers, respectively. The complex conjugate of $z \in \C$ is denoted by $\bar{z}$. For a non-zero integer $n$, we set $\zeta_{n} := \exp(2\pi \sqrt{-1} / n)$. The upper half-plane is denoted by $\HH := \{\tau \in \C \mid \mathrm{Im}(\tau) > 0\}$.

By a number field, we mean a finite extension of $\Q$. For a number field $F$, $\Z_{F}$ denotes its ring of integers. A finitely generated $\Z_{F}$-submodule of $F$ is called a fractional ideal of $F$ (or simply an ideal of $\Z_{F}$ if it is contained in $\Z_{F}$).

\section{Basic definitions}
First, we recall several foundational definitions and properties of lattices, codes, and theta functions. For details, we refer the reader to, e.g., \cite[Chs.~1 and 5]{Ebeling}, \cite{MR1695075}, \cite{MACWILLIAMS1978288}, and \cite{Dougherty1999TypeIS}.

\begin{definition}
Let $\Lambda$ be a free $\Z$-module of rank $n$, and let $b : \Lambda \times \Lambda \to \R$ be a bilinear form.
\begin{enumerate}
    \item The pair $(\Lambda, b)$ is called a \emph{lattice} if $b$ is symmetric and positive-definite, i.e., $b(x, y) = b(y, x)$ for all $x, y \in \Lambda$, and $b(x, x) > 0$ for all $x \in \Lambda \setminus \{ 0 \}$. We often denote the lattice $(\Lambda, b)$ simply by $\Lambda$.
    \item A $\Z$-submodule $\Lambda'$ of a lattice $\Lambda$ is called a \emph{sublattice} of $\Lambda$ if it is equipped with the restriction of $b$ to $\Lambda' \times \Lambda'$.
    \item A lattice $\Lambda$ is said to be \emph{integral} if $b(x,y) \in \Z$ for all $x, y \in \Lambda$.
    \item An integral lattice $\Lambda$ is called \emph{even} if $b(x,x) \in 2\Z$ for all $x \in \Lambda$, and \emph{odd} otherwise.
    
    \item A lattice $\Lambda$ is called of \emph{type $A_{n}$} if $\Lambda$ has a basis $(e_{1}, \dots, e_{n})$ such that
    \[
    b( e_{i}, e_{j} )
    = \begin{cases}
    2  & \text{if } i = j, \\
    -1 & \text{if } \lvert i - j \rvert = 1, \\
    0  & \text{otherwise}.
    \end{cases}
    \]

    \item A lattice $\Lambda$ is called of \emph{type $D_{n}$} ($n \geq 4$) if $\Lambda$ has a basis $(e_{1}, \dots, e_{n})$ such that
    \[
    b( e_{i}, e_{j} )
    = \begin{cases}
    2  & \text{if } i = j, \\
    -1 & \text{if } (\lvert i - j \rvert = 1 \text{ and } \{ i, j \} \neq \{ n-1, n \}) \text{ or } \{ i, j \} = \{ n-2, n \}, \\
    0  & \text{otherwise}.
    \end{cases}
    \]

    \item A lattice $\Lambda$ is called of \emph{type $E_{n}$} ($n = 6, 7, 8$) if $\Lambda$ has a basis $(e_{1}, \dots, e_{n})$ such that
    \[
    b( e_{i}, e_{j} )
    = \begin{cases}
    2  & \text{if } i = j, \\
    -1 & \text{if } (\lvert i - j \rvert = 1 \text{ and } \{ i, j \} \neq \{ n-1, n \}) \text{ or } \{ i, j \} = \{ n-3, n \}, \\
    0  & \text{otherwise}.
    \end{cases}
    \]

    \item For a lattice $\Lambda$, we extend $b$ $\R$-bilinearly to the vector space $V = \Lambda \otimes_{\Z} \R$. Its \emph{dual lattice} $\Lambda^{*}$ is then defined by
    \[
    \Lambda^{*} := \left\{ x \in V \mid b(x,y) \in \Z \text{ for all } y \in \Lambda \right\}.
    \]

    \item A lattice $\Lambda$ is called \emph{unimodular} if $\Lambda^* = \Lambda$.
\end{enumerate}
\end{definition}

Note that a lattice $\Lambda$ is integral if and only if $\Lambda \subset \Lambda^{*}$.




We recall the well-known structures of the quotient groups $\Lambda^{*}/\Lambda$ for lattices of type $A_{n}, D_{n}$, and $E_{n}$ as follows.

\begin{lemma}\label{ADElemma}
We have the following isomorphisms of $\Z$-modules:
\[
	\Lambda^{*}/\Lambda
	\cong \begin{cases}
		\Z/(n+1)\Z         & \text{if } \Lambda \text{ is of type } A_{n}, \\
		(\Z/2\Z)^{\oplus 2} & \text{if } \Lambda \text{ is of type } D_{n} \text{ with } n \text{ even}, \\
		\Z/4\Z              & \text{if } \Lambda \text{ is of type } D_{n} \text{ with } n \text{ odd}, \\
		\Z/(9-n)\Z          & \text{if } \Lambda \text{ is of type } E_{n}.
	\end{cases}
\]
\end{lemma}

\begin{proof}
	See, e.g., \cite[the proof of (v) $\Leftrightarrow$ (vi) in Proposition 1.5]{Ebeling}.
\end{proof}

We define three rings (or fields) of order four as follows:
\begin{itemize}
    \item $\F_2[u]/(u^2) = \{0, 1, u, 1+u\}$ with $u^2 = 0$,
    \item $\F_4 = \{0, 1, \omega, \bar{\omega}\}$ with $\bar{\omega} = \omega^2 = \omega + 1$,
    \item $\F_2 \times \F_2 = \{(0,0), (0,1), (1,0), (1,1)\}$ with componentwise addition and multiplication.
\end{itemize}

\begin{definition} Let $n \geq 2$ and $m \geq 1$ be integers.
\begin{enumerate}
	\item For $x \in (\Z/n\Z)^{\oplus m}$ and $i \in \Z/n\Z$, we denote by $\ell_i(x)$ the number of coordinates in $x$ that take the value $i$.
	\item The \emph{Euclidean weight} $\mathrm{wt}_{\mathrm{E}}(x)$ of $x \in (\Z/n\Z)^{\oplus m}$ is defined by 
	\[ 
	\mathrm{wt}_{\mathrm{E}}(x) := \sum_{i=1}^{n-1} i^2 \ell_i(x). 
	\]
	\item Let $R$ be one of the rings $\F_2[u]/(u^2)$, $\F_4$, or $\F_2 \times \F_2$. For $x = (x_1, \dots, x_m) \in R^{\oplus m}$, the \emph{Lee composition} of $x$ is defined as the triplet $(N_0(x), N_1(x), N_2(x))$, where $N_0(x)$ is the number of coordinates $x_j = 0$ (or $(0,0)$), $N_2(x)$ is the number of coordinates $x_j$ equal to $u$, $1$, or $(1,1)$, respectively, and $N_1(x) := m - N_0(x) - N_2(x)$.
	\item For $x \in R^{\oplus m}$ with $R \in \{\F_2[u]/(u^2), \F_4, \F_2 \times \F_2\}$, its \emph{Hamming weight} $\mathrm{wt}_{\mathrm{H}}(x)$, \emph{Lee weight} $\mathrm{wt}_{\mathrm{L}}(x)$, and \emph{Bachoc weight} $\mathrm{wt}_{\mathrm{B}}(x)$ are defined respectively by
	\begin{align*}
	\mathrm{wt}_{\mathrm{H}}(x) &:= N_1(x) + N_2(x), \\
	\mathrm{wt}_{\mathrm{L}}(x) &:= N_1(x) + 2N_2(x), \\
	\mathrm{wt}_{\mathrm{B}}(x) &:= 2N_1(x) + N_2(x).
	\end{align*}
	\item An overbar denotes conjugation on $\F_4$ and $\F_2 \times \F_2$; specifically, $\overline{x} = x^2$ for $x \in \F_4$, and $\overline{(a,b)} = (b,a)$ for $(a,b) \in \F_2 \times \F_2$. This notation is extended componentwise to elements and entrywise to sets; for example, $\overline{x} = (\overline{x_1}, \dots, \overline{x_m}) \in \F_4^{\oplus m}$ and $\overline{C} = \{ \overline{c} \mid c \in C \}$.
\end{enumerate}
\end{definition}

\begin{definition}
Let $n \geq 2$ and $m \geq 1$ be integers, and let $R$ be one of the rings $\Z/n\Z$, $\F_2[u]/(u^2)$, $\F_4$, or $\F_2 \times \F_2$. 
\begin{enumerate}
    \item An \emph{$R$-code $C$ of length $m$} is an $R$-submodule of $R^{\oplus m}$. The elements of $C$ are called \emph{codewords}. 
    \item The \emph{inner product} of $x = (x_1, \dots, x_m)$ and $y = (y_1, \dots, y_m)$ in $R^{\oplus m}$ is defined by
    \[
    x \cdot y := \sum_{i=1}^{m} x_i y_i.
    \]
    \item The \emph{dual code} of an $R$-code $C$ is defined by
    \[
    C^{\perp} := \left\{ x \in R^{\oplus m} \mid x \cdot y = 0 \text{ for all } y \in C \right\}.
    \]	
    \item An $R$-code $C$ is called \emph{Euclidean self-dual} if $C = C^{\perp}$. 	
    \item For $R \in \{\F_4, \F_2 \times \F_2\}$, an $R$-code $C$ is called \emph{Hermitian self-dual} if $C = \overline{C}^{\perp}$.
    \item A Euclidean self-dual code $C$ is called \emph{Type II} if it satisfies the following condition corresponding to its base ring $R$:
    \begin{itemize}
        \item $\mathrm{wt}_{\mathrm{E}}(x) \in 2n\Z$ for all $x \in C$ when $R = \Z/n\Z$,
        \item $\mathrm{wt}_{\mathrm{L}}(x) \in 4\Z$ for all $x \in C$ when $R \in \{\F_2[u]/(u^2), \F_4, \F_2 \times \F_2\}$.
    \end{itemize}
    \item A self-dual code $C$ is called \emph{Type IV} if $\mathrm{wt}_{\mathrm{H}}(x) \in 2\Z$ for all $x \in C$ under either of the following cases:
    \begin{itemize}
        \item $C$ is Euclidean self-dual over $\F_2[u]/(u^2)$,
        \item $C$ is Hermitian self-dual over $\F_2 \times \F_2$.
    \end{itemize}
\end{enumerate}
\end{definition}

%
%
%
%
%
%

\begin{definition}
Let $F$ be a number field.
\begin{enumerate}
    \item $F$ is called \emph{totally real} if the image of every field embedding $F \hookrightarrow \C$ is contained in $\R$.
    \item $F$ is called a \emph{CM field} (or a complex multiplication field) if it has no field embedding into $\R$, and there exists a totally real subfield $F^{+}$ of $F$ such that $[F : F^{+}] = 2$.
\end{enumerate}
\end{definition}

If $F$ is a CM field, the field extension $F/F^{+}$ is Galois, and its Galois group $\Gal(F/F^{+})$ is generated by the restriction of the complex conjugation to $F$.

\begin{lemma}
Let $F$ be a totally real or CM number field. Then the symmetric bilinear form 
\[
	\langle \cdot, \cdot \rangle : F \times F \to \Q, \quad (x, y) \mapsto \Tr_{F/\Q}(x\bar{y})
\]
is symmetric and positive-definite. In particular, for every fractional ideal (or finitely generated $\Z$-submodule) $\Lambda$ of $F$, the pair $(\Lambda, \langle \cdot, \cdot \rangle|_{\Lambda \times \Lambda})$ is a lattice.
\end{lemma}

\begin{proof}
The only non-trivial statement is the symmetry of $\langle \cdot, \cdot \rangle$ when $F$ is a CM field, which can be verified as follows:
\begin{align*}
	\Tr_{F/\Q}(y\bar{x}) 
	&= \Tr_{F^{+}/\Q}\left(\Tr_{F/F^{+}}(y\bar{x})\right) \\
	&= \Tr_{F^{+}/\Q}(y\bar{x} + x\bar{y}) \\
	&= \Tr_{F^{+}/\Q}\left(\Tr_{F/F^{+}}(x\bar{y})\right) \\
	&= \Tr_{F/\Q}(x\bar{y}).
\end{align*}
\end{proof}

Let $K$ be a totally real number field of degree $r = [K:\Q]$, and let $\sigma_1, \dots, \sigma_r$ be the distinct embeddings of $K$ into $\R$. We consider an $n$-dimensional vector space $V$ over $K$ equipped with a \emph{totally positive-definite} scalar product $\langle \cdot, \cdot \rangle_K$. This means that $\langle \cdot, \cdot \rangle_K : V \times V \rightarrow K$ is a symmetric bilinear form satisfying $\sigma_j(\langle v,v\rangle_K) > 0$ for all $1 \le j \le r$ and all $v \in V \setminus \{0\}$.

\begin{definition}
A $\Z_K$-submodule $\Lambda$ of $V$ is called a \emph{$\Z_K$-lattice} in $V$ if it is finitely generated and contains a $K$-basis of $V$. 
\end{definition}

For a $\Z_K$-lattice $(\Lambda, \langle \cdot, \cdot \rangle_K)$, the pair $(\Lambda, \Tr_{K/\Q}(\langle \cdot, \cdot \rangle_K))$ forms a lattice in the usual sense.

\begin{definition}
The \emph{theta function} $\theta_{\Lambda}(z)$ of a $\Z_K$-lattice $(\Lambda, \langle \cdot, \cdot \rangle_K)$ is defined for $z = (z_1, \dots, z_r) \in \HH^{\oplus r}$ by 
\[
\theta_{\Lambda}(z):=\sum_{x\in \Lambda}\exp(\pi \sqrt{-1}\Tr_{K/\Q}(\langle x,x\rangle_K z)),
\] 
where
$\Tr_{K/\Q}(\langle x,x\rangle_K z):=\displaystyle{\sum_{l=1}^{r}}\sigma_l(\langle x,x\rangle_K )z_l$.
\end{definition}

\begin{proposition}[{\cite[Theorem 5.8]{Ebeling}}]\label{K-latthe}	
If a $\Z_K$-lattice $(\Lambda, \langle \cdot, \cdot \rangle_K)$ of rank $n$ is even and unimodular as a lattice under the trace form, then its theta function $\theta_{\Lambda}(z)$ is a Hilbert modular form of weight $n/2$ for $\mathrm{SL}_2(\Z_K)$. That is, for all $M = \begin{pmatrix} \alpha & \beta \\ \gamma & \delta \end{pmatrix} \in \mathrm{SL}_2(\Z_K)$, we have
\[
\theta_{\Lambda}(M \cdot z) = N_{K/\Q}(\gamma z + \delta)^{n/2} \cdot \theta_{\Lambda}(z),
\]
where $M \cdot z := \left( \frac{\sigma_1(\alpha)z_1 + \sigma_1(\beta)}{\sigma_1(\gamma)z_1 + \sigma_1(\delta)}, \dots, \frac{\sigma_r(\alpha)z_r + \sigma_r(\beta)}{\sigma_r(\gamma)z_r + \sigma_r(\delta)} \right)$ and $N_{K/\Q}(\gamma z + \delta) := \prod_{l=1}^{r} (\sigma_l(\gamma)z_l + \sigma_l(\delta))$.
\end{proposition}

\section{Lattices and codes}\label{Sect3}
Let $\Lambda$ be a lattice of type $A_n$, $D_n$, $E_6$, or $E_7$. In this section, we construct intermediate lattices between $\Lambda^{\oplus m}$ and $(\Lambda^*)^{\oplus m}$ using codes over the finite rings corresponding to the quotient $\Lambda^*/\Lambda$. Furthermore, we investigate the structural properties of these constructed lattices, ultimately proving our main result.


\subsection{Lattices of type $A_n$ and $\Z/(n+1)\Z$-codes}
Let $\Lambda$ be a lattice of type $A_n$ with the bilinear form $b$. Then $\Lambda$ has a basis $(e_1, \dots, e_n)$ satisfying
\[
	b( e_{i}, e_{j} )
	= \begin{cases}
	2  & \text{if } i = j, \\
	-1 & \text{if } \lvert i - j \rvert = 1, \\
	0  & \text{otherwise}.
	\end{cases}
\]
We define a new set of vectors $f_i$ and $f_i^*$ for $1 \leq i \leq n$ by
\[
f_i := \sum_{l=i}^{n} e_l \quad \text{and} \quad f_i^* := f_i - \frac{1}{n+1} \sum_{l=1}^{n} f_l.
\] 	 
It is straightforward to check that the inner products of $f_i$ are given by
\[
b(f_i,f_j) =
\begin{cases}
	2 & \text{if } i = j, \\
	1 & \text{if } i \neq j.
\end{cases}
\]
Consequently, the bilinear form between $f_i$ and $f_j^*$ yields
\begin{align*}
b(f_i,f_j^*) &= b(f_i,f_j) - \frac{1}{n+1} \sum_{l=1}^n b(f_i,f_l) \\
&= \delta_{ij},
\end{align*}
where $\delta_{ij}$ denotes the Kronecker delta. Hence, $(f_1, \dots, f_n)$ forms a basis of $\Lambda$, and $(f_1^*, \dots, f_n^*)$ forms a basis of the dual lattice $\Lambda^*$.


Let $m$ be a positive integer. We define the mapping $\rho^{\oplus m} : (\Lambda^*)^{\oplus m} \longrightarrow (\Z / (n+1)\Z)^{\oplus m}$ by
\[
\rho^{\oplus m} \left( \left( \sum_{j=1}^{n} x_{ij} f^*_j \right)_{1\leq i \leq m} \right) := \left( \left( \sum_{j=1}^{n} x_{ij} \right) \bmod (n+1) \right)_{1\leq i \leq m},
\]
where $x_{ij} \in \Z$ for $1 \leq i \leq m$ and $1 \leq j \leq n$. It is clear that $\rho^{\oplus m}$ is an additive homomorphism.

Let $C$ be a $\Z/(n+1)\Z$-code of length $m$. We define the subset $\Lambda_C \subset (\Lambda^*)^{\oplus m}$ as the preimage of $C$:
\[
\Lambda_C := (\rho^{\oplus m})^{-1}(C).
\]

For $x = (x_1, \dots, x_m)$ and $y = (y_1, \dots, y_m)$ in $(\Lambda^*)^{\oplus m}$ with $x_i, y_i \in \Lambda^*$, we define a symmetric bilinear form $B$ on $(\Lambda^*)^{\oplus m}$ by
\[
B(x,y) := \sum_{i=1}^{m} b(x_i,y_i).
\]
Then, $\Lambda_C$ equipped with the bilinear form $B$ forms a lattice of rank $mn$.


We present several properties of the bilinear form $B$ and the lattice $\Lambda_C$.

\begin{lemma}\label{Anlem1}
Let $x, y \in (\Lambda^*)^{\oplus m}$. Then, $B(x,y) \in \Z$ if and only if $\rho^{\oplus m}(x) \cdot \rho^{\oplus m}(y) = 0$. Moreover, $B(x,x) \in 2\Z$ if and only if
\[
\begin{cases}
\mathrm{wt}_{\mathrm{E}}(\rho^{\oplus m}(x)) \in 2(n+1)\Z & \text{if } n \text{ is odd}, \\
\rho^{\oplus m}(x) \cdot \rho^{\oplus m}(x) = 0           & \text{if } n \text{ is even}.
\end{cases}
\]
\end{lemma}

\begin{proof}
Let $x = (x_i)_{1 \leq i \leq m}$ and $y = (y_i)_{1 \leq i \leq m}$ be elements in $(\Lambda^*)^{\oplus m}$, where $x_i = \sum_{j=1}^{n} x_{ij} f_j^*$ and $y_i = \sum_{j=1}^{n} y_{ij} f_j^*$ with $x_{ij}, y_{ij} \in \Z$ for $1 \leq i \leq m$ and $1 \leq j \leq n$. Then, the inner product $b(x_i, y_i)$ is evaluated as
\begin{align*}
	b(x_i, y_i)
	&= b\left( \sum_{j=1}^{n} x_{ij} f_j^*, \sum_{j=1}^{n} y_{ij} f_j^* \right) \\
	&= b\left( \sum_{j=1}^{n} x_{ij} f_j^*, \sum_{j=1}^{n} y_{ij} \left( f_{j} - \frac{1}{n+1} \sum_{l=1}^{n} f_{l} \right) \right) \\
	&= \sum_{j=1}^{n} x_{ij} \left( y_{ij} - \frac{1}{n+1} \sum_{l=1}^{n} y_{il} \right) \\
	&= -\frac{1}{n+1} \left( \sum_{j=1}^{n} x_{ij} \right) \left( \sum_{l=1}^{n} y_{il} \right) + \sum_{j=1}^{n} x_{ij} y_{ij}.	
\end{align*}
Summing this over all $i$, we obtain the following equivalence for $B(x,y)$:
\begin{align*}
B(x,y) \in \Z 
&\iff \frac{1}{n+1} \sum_{i=1}^{m} \left( \sum_{j=1}^{n} x_{ij} \right) \left( \sum_{l=1}^{n} y_{il} \right) \in \Z \\
&\iff \rho^{\oplus m}(x) \cdot \rho^{\oplus m}(y) = 0.
\end{align*}
In particular, for $x = y$, the local inner product becomes
\begin{align*}
	b(x_i, x_i)
	&= -\frac{1}{n+1} \left( \sum_{j=1}^{n} x_{ij} \right)^2 + \sum_{j=1}^{n} x_{ij}^2 \\
	&= -\frac{1}{n+1} \left( \sum_{j=1}^{n} x_{ij} \right)^2 + \sum_{j=1}^{n} x_{ij}^2 + \left( \sum_{j=1}^{n} x_{ij} \right)^2 - \left( \sum_{j=1}^{n} x_{ij} \right)^2 \\
	&= \frac{n}{n+1} \left( \sum_{j=1}^{n} x_{ij} \right)^2 - \sum_{1 \leq j < l \leq n} 2 x_{ij} x_{il}.
\end{align*}
Since the second term $- \sum 2 x_{ij} x_{il}$ is always an even integer, we have
\begin{align*}
B(x,x) \in 2\Z 
&\iff \frac{n}{n+1} \sum_{i=1}^{m} \left( \sum_{j=1}^{n} x_{ij} \right)^2 \in 2\Z \\
&\iff \begin{cases}
\mathrm{wt}_{\mathrm{E}}(\rho^{\oplus m}(x)) \in 2(n+1)\Z & \text{if } n \text{ is odd}, \\
\rho^{\oplus m}(x) \cdot \rho^{\oplus m}(x) = 0           & \text{if } n \text{ is even}.
\end{cases}
\end{align*}
\end{proof}

\begin{lemma}\label{Anlem2}
Let $C$ be a $\Z/(n+1)\Z$-code of length $m$. Then, the dual lattice of $\Lambda_C$ is given by
\[
\Lambda_C^* = \Lambda_{C^{\perp}}.
\]	
\end{lemma}

\begin{proof}
Let $x \in \Lambda_{C^{\perp}}$ and $y \in \Lambda_{C}$. Then, by definition, $\rho^{\oplus m}(x) \in C^{\perp}$ and $\rho^{\oplus m}(y) \in C$, which implies $\rho^{\oplus m}(x) \cdot \rho^{\oplus m}(y) = 0$. By Lemma \ref{Anlem1}, we have $B(x,y) \in \Z$, which establishes the inclusion $\Lambda_{C^{\perp}} \subset \Lambda_{C}^*$.

Conversely, let $x \in \Lambda_{C}^*$ and $y \in \Lambda_C$. Since $\Lambda^{\oplus m} \subset \Lambda_C \subset (\Lambda^{*})^{\oplus m}$, taking the dual yields the chain of inclusions $\Lambda^{\oplus m} \subset \Lambda_{C}^* \subset (\Lambda^{*})^{\oplus m}$. Thus, $x \in (\Lambda^{*})^{\oplus m}$ and $B(x,y) \in \Z$ for all $y \in \Lambda_C$. Lemma \ref{Anlem1} then implies that $\rho^{\oplus m}(x) \cdot \rho^{\oplus m}(y) = 0$ for all $\rho^{\oplus m}(y) \in C$. Since $\rho^{\oplus m}$ is surjective, this means $\rho^{\oplus m}(x) \in C^{\perp}$, and hence $x \in \Lambda_{C^{\perp}}$. Consequently, $\Lambda_{C}^* \subset \Lambda_{C^{\perp}}$, completing the proof.
\end{proof}


From Lemmas \ref{Anlem1} and \ref{Anlem2}, we obtain the following characterization of the lattice $\Lambda_C$.
\begin{theorem}\label{Anthe}
Let $C$ be a $\Z/(n+1)\Z$-code of length $m$.
\begin{enumerate}
    \item[\rm (I)] When $n$ is odd:
    \begin{enumerate}
        \item $\Lambda_{C}$ is integral if and only if $C \subset C^{\perp}$.
        \item $\Lambda_{C}$ is unimodular if and only if $C$ is Euclidean self-dual. 
        \item $\Lambda_{C}$ is even if and only if $\mathrm{wt}_{\mathrm{E}}(x) \in 2(n+1)\Z$ for all $x \in C$.
        \item $\Lambda_{C}$ is even unimodular if and only if $C$ is Type II.
    \end{enumerate}
    \item[\rm (II)] When $n$ is even:
    \begin{enumerate}
        \item $\Lambda_{C}$ is integral if and only if $C \subset C^{\perp}$.
        \item $\Lambda_{C}$ is unimodular if and only if $C$ is Euclidean self-dual. 
        \item $\Lambda_{C}$ is even if and only if $C \subset C^{\perp}$.
        \item $\Lambda_{C}$ is even unimodular if and only if $C$ is Euclidean self-dual.
    \end{enumerate}
\end{enumerate}	
\end{theorem}

Since the rank of an even unimodular lattice must be divisible by $8$ (see, e.g., \cite[Corollary~18 in Chapter~7]{conway1998sphere}), we immediately obtain the following corollary.

\begin{corollary}\label{Ancor}
Let $C$ be a $\Z/(n+1)\Z$-code of length $m$.
\begin{enumerate}
    \item If $n$ is odd and $C$ is Type II, then $mn \in 8\Z$.
    \item If $n$ is even and $C$ is Euclidean self-dual, then $mn \in 8\Z$.
\end{enumerate}	
\end{corollary}

\subsection{Lattices of type $D_n$ and $\Z/4\Z$-codes}
Let $n \geq 4$ be an odd integer, and let $\Lambda$ be a lattice of type $D_n$ with the bilinear form $b$. Then $\Lambda$ has a basis $(e_1, \dots, e_n)$ satisfying
\[
	b( e_{i}, e_{j} )
	= \begin{cases}
	2  & \text{if } i = j, \\
	-1 & \text{if } (\lvert i - j \rvert = 1 \text{ and } \{ i, j \} \neq \{ n-1, n \}) \text{ or } \{ i, j \} = \{ n-2, n \}, \\
	0  & \text{otherwise}.
	\end{cases}
\]
We define a new set of vectors $f_i$ for $1 \leq i \leq n$ by
\begin{align*}
f_i &:= e_i + e_{n-1} + e_{n} + \sum_{l=i+1}^{n-2} 2e_{l} \quad (1 \leq i \leq n-3), \\
f_{n-2} &:= e_{n-2} + e_{n-1} + e_{n}, \quad f_{n-1} := e_n, \quad f_n := e_n + \sum_{l=1}^{n-2} e_{l},
\end{align*}
and their dual candidates $f_i^*$ for $1 \leq i \leq n$ by
\[
f_i^* := \frac{1}{4} \sum_{l=0}^{n-1} (-1)^l (n-2l) f_{i+l},
\]
where the indices of $f_{i+l}$ are considered modulo $n$ (taking values in $\{1, \dots, n\}$). It is straightforward to verify that the inner products among $f_i$ are given by
\[
	b( f_{i}, f_{j} )
	= \begin{cases}
	2 & \text{if } i = j, \\
	1 & \text{if } \lvert i - j \rvert = 1 \text{ or } \{ i, j \} = \{ 1, n \}, \\
	0 & \text{otherwise}.
	\end{cases}
\]
Consequently, the bilinear form between $f_i$ and $f_j^*$ yields
\begin{align*}
b(f_i,f_j^*)
&= \begin{dcases}
\frac{1}{4} \sum_{l \in \{i-j-1, i-j, i-j+1\}} (-1)^{l}(n-2l) b(f_i,f_{j+l}) & \text{if } i - j \geq 1, \\
\frac{1}{4} \sum_{l \in \{n-1, 0, 1\}} (-1)^{l}(n-2l) b(f_i,f_{j+l})       & \text{if } i = j, \\
\frac{1}{4} \sum_{l \in \{n-2, n-1, 0\}} (-1)^{l}(n-2l) b(f_i,f_{j+l})     & \text{if } i - j = -1, \\
\frac{1}{4} \sum_{l \in \{i-j-1+n, i-j+n, i-j+1+n\}} (-1)^{l}(n-2l) b(f_i,f_{j+l}) & \text{if } i - j < -1
\end{dcases} \\
&= \delta_{ij}.
\end{align*}
Hence, $(f_1, \dots, f_n)$ forms a basis of $\Lambda$, and $(f_1^*, \dots, f_n^*)$ forms a basis of the dual lattice $\Lambda^*$.


Let $m$ be a positive integer. We define the mapping $\rho^{\oplus m} : (\Lambda^*)^{\oplus m} \longrightarrow (\Z / 4\Z)^{\oplus m}$ by
\[
\rho^{\oplus m} \left( \left( \sum_{j=1}^{n} x_{ij} f^*_j \right)_{1\leq i \leq m} \right) := \left( \left( \sum_{j=1}^{n} x_{ij} \right) \bmod 4 \right)_{1\leq i \leq m},
\]
where $x_{ij} \in \Z$ for $1 \leq i \leq m$ and $1 \leq j \leq n$. It is clear that $\rho^{\oplus m}$ is an additive homomorphism.

Let $C$ be a $\Z/4\Z$-code of length $m$. We define the subset $\Lambda_C \subset (\Lambda^*)^{\oplus m}$ as the preimage of $C$:
\[
\Lambda_C := (\rho^{\oplus m})^{-1}(C).
\]

For $x = (x_1, \dots, x_m)$ and $y = (y_1, \dots, y_m)$ in $(\Lambda^*)^{\oplus m}$ with $x_i, y_i \in \Lambda^*$, we define a symmetric bilinear form $B$ on $(\Lambda^*)^{\oplus m}$ by
\[
B(x,y) := \sum_{i=1}^{m} b(x_i,y_i).
\]
Then, $\Lambda_C$ equipped with the bilinear form $B$ forms a lattice of rank $mn$.



We present several properties of the bilinear form $B$ and the lattice $\Lambda_C$.

\begin{lemma}\label{Dnlem1}
Let $x,y \in (\Lambda^*)^{\oplus m}$. Then, $B(x,y) \in \Z$ if and only if $\rho^{\oplus m}(x) \cdot \rho^{\oplus m}(y) = 0$. Moreover, $B(x,x) \in 2\Z$ if and only if $\mathrm{wt}_{\mathrm{E}}(\rho^{\oplus m}(x)) \in 8\Z$.
\end{lemma}

\begin{proof}
Let $x = (x_i)_{1 \leq i \leq m}$ and $y = (y_i)_{1 \leq i \leq m}$ be elements in $(\Lambda^*)^{\oplus m}$, where $x_i = \sum_{j=1}^{n} x_{ij} f_j^*$ and $y_i = \sum_{j=1}^{n} y_{ij} f_j^*$ with $x_{ij}, y_{ij} \in \Z$ for $1 \leq i \leq m$ and $1 \leq j \leq n$. 

Since the dual basis expanded under the primal basis is expressed as
\[
\sum_{j=1}^{n} y_{ij} f^*_j 
= \sum_{j=1}^{n} y_{ij} \left( \frac{1}{4} \sum_{l=0}^{n-1} (-1)^l (n-2l) f_{j+l} \right) 
=\sum^{n}_{s=1} \frac{1}{4}\left( \sum^{n}_{t=1}(-1)^{n-t}(n-2(n-t))y_{i(s+t)} \right)f_s,
\]
where the indices are considered modulo $n$ (taking values in $\{1, \dots, n\}$), the local inner product $b(x_i, y_i)$ can be computed as follows:
\begin{align*}
	b(x_i, y_i)
	&= b\left( \sum_{j=1}^{n} x_{ij} f_j^*, \sum_{j=1}^{n} y_{ij} f_j^* \right) \\
	&= b\left( \sum_{j=1}^{n} x_{ij} f_j^*, \sum_{j=1}^{n} \frac{1}{4} \left( \sum_{l=1}^{n} (-1)^{n-l} (n-2(n-l)) y_{i(j+l)} \right) f_j \right) \\
	&= \sum_{j=1}^{n} x_{ij} \frac{1}{4} \left( \sum_{l=1}^{n} (-1)^{n-l} (n-2(n-l)) y_{i(j+l)} \right) \\
	&= \frac{n}{4} \left( \sum_{j=1}^{n} x_{ij} \right) \left( \sum_{j=1}^{n} y_{ij} \right) - \frac{n}{4} \left( \sum_{j=1}^{n} x_{ij} \right) \left( \sum_{j=1}^{n} y_{ij} \right) \\
	&\quad + \sum_{j=1}^{n} x_{ij} \frac{1}{4} \left( \sum_{l=1}^{n} (-1)^{n-l} (2l-n) y_{i(j+l)} \right) \\
	&= \frac{n}{4} \left( \sum_{j=1}^{n} x_{ij} \right) \left( \sum_{j=1}^{n} y_{ij} \right) + \sum_{j=1}^n x_{ij} \frac{1}{4} \left( \sum_{l=1}^{n} \left( -n y_{i(j+l)} + (-1)^{n-l} (2l-n) y_{i(j+l)} \right) \right) \\
	&= \frac{n}{4} \left( \sum_{j=1}^{n} x_{ij} \right) \left( \sum_{j=1}^{n} y_{ij} \right) + \sum_{j=1}^n x_{ij} \frac{1}{4} \left( \sum_{l=1}^{(n+1)/2} (4l - 2(n+1)) y_{i(j+2l-1)} - \sum_{l=1}^{(n-1)/2} 4l y_{i(j+2l)} \right).
\end{align*}
Summing this over all $1 \leq i \leq m$, we secure the following equivalence for $B(x,y) \in \Z$:
\begin{align*}
B(x,y) \in \Z 
&\iff \frac{n}{4} \sum_{i=1}^{m} \left( \sum_{j=1}^{n} x_{ij} \right) \left( \sum_{l=1}^{n} y_{il} \right) \in \Z \\
&\iff \rho^{\oplus m}(x) \cdot \rho^{\oplus m}(y) = 0.
\end{align*}
Furthermore, for the quadratic form part, direct evaluation yields
\begin{align*}
	&\sum_{j=1}^n x_{ij} \frac{1}{4} \left( \sum_{l=1}^{(n+1)/2} (4l - 2(n+1)) x_{i(j+2l-1)} - \sum_{l=1}^{(n-1)/2} 4l x_{i(j+2l)} \right) \\
	&= \sum_{j=1}^{n} x_{ij} \frac{1}{4} \left( 4 \left( \frac{n+1}{2} \right) - 2(n+1) \right) x_{ij} \\
	&\quad + \sum_{\substack{1 \leq s < t \leq n \\ t-s \in 2\Z}} \frac{1}{4} \left( \left( 4 \frac{n+1-(t-s)}{2} - 2(n+1) \right) - 4 \frac{t-s}{2} \right) x_{is} x_{it} \\
	&\quad + \sum_{\substack{1 \leq s < t \leq n \\ t-s \notin 2\Z}} \frac{1}{4} \left( \left( 4 \frac{(t-s)+1}{2} - 2(n+1) \right) - 4 \frac{n-(t-s)}{2} \right) x_{is} x_{it} \\
	&= \sum_{\substack{1 \leq s < t \leq n \\ t-s \in 2\Z}} (s-t) x_{is} x_{it} + \sum_{\substack{1 \leq s < t \leq n \\ t-s \notin 2\Z}} ((t-s)-n) x_{is} x_{it},
\end{align*}
which always belongs to $2\Z$. Thus, the evenness condition reduces to
\begin{align*}
B(x,x) \in 2\Z 
&\iff \frac{n}{4} \sum_{i=1}^{m} \left( \sum_{j=1}^{n} x_{ij} \right)^2 \in 2\Z \\
&\iff \mathrm{wt}_{\mathrm{E}}(\rho^{\oplus m}(x)) \in 8\Z,
\end{align*}
completing the proof.
\end{proof}

\begin{lemma}\label{Dnlem2}
Let $C$ be a $\Z/4\Z$-code of length $m$. Then, the dual lattice of $\Lambda_C$ is given by
\[
\Lambda_C^* = \Lambda_{C^{\perp}}.
\]	
\end{lemma}

\begin{proof}
The assertion can be proven in the exact same manner as Lemma \ref{Anlem2}.
\end{proof}

From Lemmas \ref{Dnlem1} and \ref{Dnlem2}, we obtain the following theorem.

\begin{theorem}\label{Dnthe}
Let $C$ be a $\Z/4\Z$-code of length $m$.
\begin{enumerate}
    \item $\Lambda_{C}$ is integral if and only if $C \subset C^{\perp}$.
    \item $\Lambda_{C}$ is unimodular if and only if $C$ is Euclidean self-dual. 
    \item $\Lambda_{C}$ is even if and only if $\mathrm{wt}_{\mathrm{E}}(x) \in 8\Z$ for all $x \in C$.
    \item $\Lambda_{C}$ is even unimodular if and only if $C$ is Type II.
\end{enumerate}	
\end{theorem}

\subsection{Lattices of type $D_n$ and $\F_2[u]/(u^2)$-codes}\label{Dnsub2}
Let $n \geq 4$ be an even integer, and let $\Lambda$ be a lattice of type $D_n$ with the bilinear form $b$. Then $\Lambda$ has a basis $(e_1, \dots, e_n)$ such that 
\[
	b( e_{i}, e_{j} )
	= \begin{cases}
	2  & \text{if } i = j, \\
	-1 & \text{if } (\lvert i - j \rvert = 1 \text{ and } \{ i, j \} \neq \{ n-1, n \}) \text{ or } \{ i, j \} = \{ n-2, n \}, \\
	0  & \text{otherwise}.
	\end{cases}
\]
We define a new set of vectors $f_i$ for $1 \leq i \leq n$ by
\begin{align*}
f_i &:= e_i \quad (1 \leq i \leq n-1), \\
f_n &:= e_n + \sum_{l=1}^{n-2} e_l,
\end{align*}
and their dual candidates $f_i^*$ for $1 \leq i \leq n$ by
\[
f_i^* := \sum_{l=1}^{n} \frac{1}{4} (n - 2\lvert i-l \rvert) f_l,
\]
where the indices of $f_l$ are considered modulo $n$ (taking values in $\{1, \dots, n\}$). It is straightforward to verify that the inner products among $f_i$ are given by
\begin{align*}
	b( f_{i}, f_{j} )
	= \begin{cases}
	2  & \text{if } i = j, \\
	-1 & \text{if } \lvert i - j \rvert = 1, \\
	1  & \text{if } \{i,j\} = \{1,n\}, \\
	0  & \text{otherwise}.
	\end{cases}
\end{align*}
Consequently, the bilinear form between $f_i$ and $f_j^*$ yields
\begin{align*}	
b(f_i,f_j^*)
&= \begin{dcases}
	\sum_{l=i-1}^{i+1} \frac{1}{4} (n - 2\lvert i-l \rvert) b(f_i,f_l) = 1
	& \text{if } 1 < i = j < n, \\
	\sum_{l \in \{n,1,2\}} \frac{1}{4} (n - 2\lvert 1-l \rvert) b(f_1,f_l) = 1
	& \text{if } i = j = 1, \\
	\sum_{l \in \{n-1,n,1\}} \frac{1}{4} (n - 2\lvert n-l \rvert) b(f_n,f_l) = 1
	& \text{if } i = j = n, \\
	\sum_{l=i-1}^{i+1} \frac{1}{4} (n - 2\lvert j-l \rvert) b(f_i,f_l) = 0
	& \text{if } 1 < i < n \text{ and } i \neq j, \\
	\sum_{l \in \{n,1,2\}} \frac{1}{4} (n - 2\lvert j-l \rvert) b(f_1,f_l) = 0
	& \text{if } j \neq i = 1, \\
	\sum_{l \in \{n-1,n,1\}} \frac{1}{4} (n - 2\lvert j-l \rvert) b(f_n,f_l) = 0
	& \text{if } j \neq i = n
\end{dcases} \\[1ex]
&= \delta_{ij},
\end{align*}
where $\delta_{ij}$ denotes the Kronecker delta. Hence, $(f_1, \dots, f_n)$ forms a basis of $\Lambda$, and $(f_1^*, \dots, f_n^*)$ forms a basis of the dual lattice $\Lambda^*$.


Let $m$ be a positive integer. We define the mapping $\rho^{\oplus m} : (\Lambda^*)^{\oplus m} \longrightarrow (\F_2[u]/(u^2))^{\oplus m}$ by
\[
\rho^{\oplus m} \left( \left( \sum_{j=1}^{n} x_{ij} f^*_j \right)_{1\leq i \leq m} \right) := \left( \left( \sum_{j=1}^{n/2} x_{i(2j-1)} \bmod 2 \right) + \left( \sum_{j=1}^{n/2} x_{i(2j)} \bmod 2 \right)(1+u) \right)_{1\leq i \leq m},
\]
where $x_{ij} \in \Z$ for $1 \leq i \leq m$ and $1 \leq j \leq n$. It is clear that $\rho^{\oplus m}$ is an additive homomorphism.

Let $C$ be an $\F_2[u]/(u^2)$-code of length $m$. We define the subset $\Lambda_C \subset (\Lambda^*)^{\oplus m}$ as the preimage of $C$:
\[
\Lambda_C := (\rho^{\oplus m})^{-1}(C).
\]

For $x = (x_1, \dots, x_m)$ and $y = (y_1, \dots, y_m)$ in $(\Lambda^*)^{\oplus m}$ with $x_i, y_i \in \Lambda^*$, we define a symmetric bilinear form $B$ on $(\Lambda^*)^{\oplus m}$ by
\[
B(x,y) := \sum_{i=1}^{m} b(x_i,y_i).
\]
Then, $\Lambda_C$ equipped with the bilinear form $B$ forms a lattice of rank $mn$.

We present several properties of the bilinear form $B$ and the lattice $\Lambda_C$.

\begin{lemma}\label{Dn2lem1}
Let $x,y \in (\Lambda^*)^{\oplus m}$.
\begin{enumerate}
    \item[\rm (I)] When $n \notin 4\Z$:
    \begin{align*}
    B(x,y) \in \Z &\iff \rho^{\oplus m}(x) \cdot \rho^{\oplus m}(y) \in \{ 0, 1+u \}, \\
    B(x,x) \in 2\Z &\iff \mathrm{wt}_{\mathrm{L}}(\rho^{\oplus m}(x)) \in 4\Z.
    \end{align*}
    \item[\rm (II)] When $n \in 4\Z$:
    \begin{align*}
    B(x,y) \in \Z &\iff \rho^{\oplus m}(x) \cdot \rho^{\oplus m}(y) \in \{0, 1\}, \\
    B(x,x) \in 2\Z &\iff \begin{cases}
    \mathrm{wt}_{\mathrm{H}}(\rho^{\oplus m}(x)) \in 2\Z & \text{if } n \notin 8\Z, \\
    \mathrm{wt}_{\mathrm{B}}(\rho^{\oplus m}(x)) \in 2\Z & \text{if } n \in 8\Z.
    \end{cases}
    \end{align*}
\end{enumerate}
\end{lemma}


\begin{proof}
Let $x=(x_i)_{1\leq i \leq m}$ and $y=(y_i)_{1\leq i \leq m}$ be elements in $(\Lambda^*)^{\oplus m}$, where $x_i=\sum_{j=1}^{n}x_{ij}f_j^*$ and $y_i=\sum^{n}_{j=1} y_{ij}f_j^*$ with $x_{ij}, y_{ij}\in \Z$ for $1 \leq i \leq m$ and $1 \leq j \leq n$. Then, we have
\begin{align}
	b(x_i, y_i)
	&= b\left( \sum_{j=1}^{n}x_{ij}f_j^*, \sum_{j=1}^{n}y_{ij}f_j^* \right) \notag \\
	&= b\left( \sum_{j=1}^{n}x_{ij}f_j^*, \sum^{n}_{j=1} \frac{1}{4}y_{ij}\left( \sum^{n}_{l=1}(n-2|j-l|)f_l \right) \right) \notag \\
	&= \sum_{j=1}^{n}\sum_{s=1}^{n}\frac{1}{4}(n-2|s-j|)x_{ij}y_{is} \notag \\
	&= \sum_{j=1}^{n/2}\sum_{s=1}^{n/2}\left(\frac{n}{4}-|s-j|\right)x_{i(2j-1)}y_{i(2s-1)} + \sum_{j=1}^{n/2}\sum_{s=1}^{n/2}\left(\frac{n}{4}-|s-j|\right)x_{i(2j)}y_{i(2s)} \notag \\
	&\quad + \sum_{j=1}^{n/2}\sum_{s=1}^{n/2}\left(\frac{n}{4}-\left|s-j+\frac{1}{2}\right|\right)x_{i(2j-1)}y_{i(2s)} + \sum_{j=1}^{n/2}\sum_{s=1}^{n/2}\left(\frac{n}{4}-\left|s-j-\frac{1}{2}\right|\right)x_{i(2j)}y_{i(2s-1)} \notag \\
	&= \sum_{j=1}^{n/2}\sum_{s=1}^{n/2}\left(\frac{n}{4}-|s-j|\right)x_{i(2j-1)}y_{i(2s-1)} + \sum_{j=1}^{n/2}\sum_{s=1}^{n/2}\left(\frac{n}{4}-|s-j|\right)x_{i(2j)}y_{i(2s)} \notag \\
	&\quad + \sum_{j=1}^{n/2}\sum_{s=1}^{n/2}\left(\frac{n}{4}-\frac{1}{2}-\left|s-j\right|\right)x_{i(2j-1)}y_{i(2s)} + \sum_{j=1}^{n/2}\sum_{s=1}^{n/2}\left(\frac{n}{4}-\frac{1}{2}-|s-j|\right)x_{i(2j)}y_{i(2s-1)} \notag \\
	&\quad + \sum_{j=1}^{(n/2)-1}(x_{i(2j+1)}y_{i(2j)}+x_{i(2j)}y_{i(2j+1)}). \label{Eq:Dn} 
\end{align}

Let $n \notin 4\Z$. Since $n \in 2\Z$, we see that
\begin{align*}
	&\sum_{j=1}^{n/2}\sum_{s=1}^{n/2}(-|s-j|)x_{i(2j-1)}y_{i(2s-1)} + \sum_{j=1}^{n/2}\sum_{s=1}^{n/2}(-|s-j|)x_{i(2j)}y_{i(2s)} \\
	&+ \sum_{j=1}^{n/2}\sum_{s=1}^{n/2}\left(\frac{n}{4}-\frac{1}{2}-\left|s-j\right|\right)x_{i(2j-1)}y_{i(2s)} + \sum_{j=1}^{n/2}\sum_{s=1}^{n/2}\left(\frac{n}{4}-\frac{1}{2}-|s-j|\right)x_{i(2j)}y_{i(2s-1)} \\
	&+ \sum_{j=1}^{(n/2)-1}(x_{i(2j+1)}y_{i(2j)}+x_{i(2j)}y_{i(2j+1)}) \in \Z.
\end{align*}
Thus, we obtain
\[
b(x_i,y_i) \in \frac{1}{2} \left( \left(\sum_{j=1}^{n/2}x_{i(2j-1)}\right) \left(\sum_{s=1}^{n/2}y_{i(2s-1)}\right) + \left(\sum_{j=1}^{n/2}x_{i(2j)}\right) \left(\sum_{s=1}^{n/2}y_{i(2s)}\right) \right) + \Z.
\]
By considering the parity of $\sum_{j=1}^{n/2}x_{i(2j-1)}$, $\sum_{j=1}^{n/2}x_{i(2j)}$, $\sum_{j=1}^{n/2}y_{i(2j-1)}$, and $\sum_{j=1}^{n/2}y_{i(2j)}$, which corresponds to the pair $(\rho^{\oplus 1}(x_i), \rho^{\oplus 1}(y_i))$, we find that
\[
b(x_i,y_i) \in \begin{dcases}
	\Z & \text{if } \rho^{\oplus 1}(x_i) \cdot \rho^{\oplus 1}(y_i) \in \{ 0,1+u\}, \\
	\frac{1}{2}+\Z & \text{if } \rho^{\oplus 1}(x_i) \cdot \rho^{\oplus 1}(y_i) \in \{ 1,u\}.
\end{dcases}
\]
Hence,
\[
B(x,y) \in \Z \iff \rho^{\oplus m}(x) \cdot \rho^{\oplus m}(y) \in \{0,1+u\}.
\]

Let $n \in 4\Z$. We see that
\begin{align*}
	&\sum_{j=1}^{n/2}\sum_{s=1}^{n/2}\left(\frac{n}{4}-|s-j|\right)x_{i(2j-1)}y_{i(2s-1)} + \sum_{j=1}^{n/2}\sum_{s=1}^{n/2}\left(\frac{n}{4}-|s-j|\right)x_{i(2j)}y_{i(2s)} \\
	&+ \sum_{j=1}^{n/2}\sum_{s=1}^{n/2}\left(\frac{n}{4}-\left|s-j\right|\right)x_{i(2j-1)}y_{i(2s)} + \sum_{j=1}^{n/2}\sum_{s=1}^{n/2}\left(\frac{n}{4}-|s-j|\right)x_{i(2j)}y_{i(2s-1)} \\
	&+ \sum_{j=1}^{(n/2)-1}(x_{i(2j+1)}y_{i(2j)}+x_{i(2j)}y_{i(2j+1)}) \in \Z.
\end{align*}
Thus,
\[
b(x_i,y_i) \in \frac{1}{2} \left( \left(\sum_{j=1}^{n/2}x_{i(2j-1)}\right) \left(\sum_{s=1}^{n/2}y_{i(2s)}\right) + \left(\sum_{j=1}^{n/2}x_{i(2j)}\right) \left(\sum_{s=1}^{n/2}y_{i(2s-1)}\right) \right) + \Z.
\]
By considering the pair $(\rho^{\oplus 1}(x_i), \rho^{\oplus 1}(y_i))$, we find that
\[
b(x_i,y_i) \in \begin{dcases}
	\Z & \text{if } \rho^{\oplus 1}(x_i) \cdot \rho^{\oplus 1}(y_i) \in \{ 0,1\}, \\
	\frac{1}{2}+\Z & \text{if } \rho^{\oplus 1}(x_i) \cdot \rho^{\oplus 1}(y_i) \in \{ 1+u,u\}.
\end{dcases}
\]
Hence,
\[
B(x,y) \in \Z \iff \rho^{\oplus m}(x) \cdot \rho^{\oplus m}(y) \in \{0,1\}.
\]

Moreover, for the quadratic form, we have
\begin{align*}
b(x_i,x_i) 
&= \frac{n}{4}\sum_{j=1}^{n/2}\sum_{s=1}^{n/2}x_{i(2j-1)}x_{i(2s-1)} + \frac{n}{4}\sum_{j=1}^{n/2}\sum_{s=1}^{n/2}x_{i(2j)}x_{i(2s)} \\
	&\quad - 2\sum_{1\leq j<s\leq n/2}|s-j|(x_{i(2j-1)}x_{i(2s-1)} + x_{i(2j)}x_{i(2s)}) \\
	&\quad + 2\sum_{j=1}^{n/2}\sum_{s=1}^{n/2}\left(\frac{n}{4}-\frac{1}{2}-\left|s-j\right|\right)x_{i(2j-1)}x_{i(2s)} + 2\sum_{j=1}^{(n/2)-1}x_{i(2j+1)}x_{i(2j)}.	
\end{align*}

Let $n \notin 4\Z$. Since $n \in 2\Z$, we see that
\[
b(x_i,x_i) \in \frac{n}{4} \left( \left(\sum_{j=1}^{n/2}x_{i(2j-1)}\right)^2 + \left(\sum_{j=1}^{n/2}x_{i(2j)}\right)^2 \right) + 2\Z. 
\]
Thus,
\[
2b(x_i,x_i) \equiv \begin{cases}
	0 \pmod 4 & \text{if } \rho^{\oplus 1}(x_i)=0, \\
	1 \pmod 4 & \text{if } \rho^{\oplus 1}(x_i)=1 \text{ or } 1+u, \\
	2 \pmod 4 & \text{if } \rho^{\oplus 1}(x_i)=u. 
\end{cases}	
\]
Hence,
\[
B(x,x) \in 2\Z \iff \mathrm{wt}_{\mathrm{L}}(\rho^{\oplus m}(x)) \in 4\Z.
\]

Let $n \in 4\Z$ and $n \notin 8\Z$. We see that 
\[
b(x_i,x_i) \in \frac{n}{4} \left( \left(\sum_{j=1}^{n/2}x_{i(2j-1)}\right)^2 + \left(\sum_{j=1}^{n/2}x_{i(2j)}\right)^2 \right) - \left(\sum_{j=1}^{n/2}x_{i(2j-1)}\right)\left(\sum_{j=1}^{n/2}x_{i(2j)}\right) + 2\Z.
\]
Thus,
\[
b(x_i,x_i) \equiv \begin{cases}
	0 \pmod 2 & \text{if } \rho^{\oplus 1}(x_i)=0, \\
	1 \pmod 2 & \text{otherwise}. 
\end{cases}	
\]
Hence,
\[
B(x,x) \in 2\Z \iff \mathrm{wt}_{\mathrm{H}}(\rho^{\oplus m}(x)) \in 2\Z.
\]

Let $n \in 8\Z$. We see that 
\[
b(x_i,x_i) \in \left(\sum_{j=1}^{n/2}x_{i(2j-1)}\right)\left(\sum_{j=1}^{n/2}x_{i(2j)}\right) + 2\Z.
\]
Thus,
\[
b(x_i,x_i) \equiv \begin{cases}
	1 \pmod 2 & \text{if } \rho^{\oplus 1}(x_i)=u, \\
	0 \pmod 2 & \text{otherwise}. 
\end{cases}	
\]
Hence,
\[
B(x,x) \in 2\Z \iff \mathrm{wt}_{\mathrm{B}}(\rho^{\oplus m}(x)) \in 2\Z.
\]
\end{proof}


\begin{lemma}\label{Dn2lem2}
Let $C$ be an $\F_2[u]/(u^2)$-code of length $m$. We define
\[
C' := \{ x \in (\F_2[u]/(u^2))^{\oplus m} \mid x \cdot y \in \{0,1\} \text{ for all } y \in C \},
\]	
and
\[
C'' := \{ x \in (\F_2[u]/(u^2))^{\oplus m} \mid x \cdot y \in \{0,1+u\} \text{ for all } y \in C \}.
\]
Then, $C'=C^{\perp}$ and $C''=C^{\perp}$ hold.
\end{lemma}

\begin{proof}
First, we show $C'=C^{\perp}$. The inclusion $C^{\perp} \subset C'$ can be easily checked. Let $x \in C'$. Suppose that there exists $y \in C$ such that $x \cdot y = 1$. Then, we have $x \cdot (1+u) y = 1+u$. Since $(1+u) y \in C$, it follows that $x \cdot (1+u) y \in \{0,1\}$, which is a contradiction. Hence, we have $x \cdot y = 0$ for all $y \in C$. Therefore, $C'=C^{\perp}$ holds. In a similar manner, we can show $C''=C^{\perp}$.   
\end{proof}

From Lemma \ref{Dn2lem2}, we can determine the dual lattice of $\Lambda_C$ in the same way as Lemma \ref{Anlem2}.

\begin{lemma}\label{Dn2lem3}
Let $C$ be an $\F_2[u]/(u^2)$-code of length $m$. Then,
\[
\Lambda_C^* = \Lambda_{C^{\perp}}.
\]
\end{lemma}

\begin{proof}
Let $x \in \Lambda_{C^{\perp}}$ and $y \in \Lambda_{C}$. Then, $\rho^{\oplus m}(x) \in C^{\perp}$ and $\rho^{\oplus m}(y) \in C$, which implies $\rho^{\oplus m}(x) \cdot \rho^{\oplus m}(y) = 0$. By Lemma \ref{Dn2lem1}, we have $B(x,y) \in \Z$. Thus, $\Lambda_{C^{\perp}} \subset \Lambda_{C}^*$ holds.

Conversely, let $x \in \Lambda_{C}^*$ and $y \in \Lambda_C$. From $(\Lambda)^{\oplus m} \subset \Lambda_C \subset (\Lambda^{*})^{\oplus m}$, we see that $(\Lambda)^{\oplus m} \subset \Lambda_{C}^* \subset (\Lambda^{*})^{\oplus m}$. Then, $x \in \Lambda_{C}^* \subset (\Lambda^{*})^{\oplus m}$ and $B(x,y) \in \Z$. By Lemma \ref{Dn2lem1}, we obtain
\begin{align*}
\rho^{\oplus m}(x) \cdot \rho^{\oplus m}(y) \in
\begin{dcases}
	\{0,1+u\} & \text{if } n \notin 4\Z, \\
	\{0,1\}   & \text{if } n \in 4\Z.
\end{dcases}
\end{align*}
Hence, we have
\begin{align*}
\Lambda_{C}^* \subset 
\begin{cases}
	\Lambda_{C''} & \text{if } n \notin 4\Z, \\
	\Lambda_{C'}  & \text{if } n \in 4\Z.
\end{cases}
\end{align*}
From Lemma \ref{Dn2lem2}, we conclude that $\Lambda_{C}^* \subset \Lambda_{C^{\perp}}$.
\end{proof}

From Lemmas \ref{Dn2lem1} and \ref{Dn2lem3}, we obtain the following classification theorem.

\begin{theorem}\label{Dn2the1}
Let $C$ be an $\F_2[u]/(u^2)$-code of length $m$.
\begin{enumerate}
	\item[\rm (I)] When $n \in 2\Z$ and $n \notin 4\Z$:
	\begin{enumerate}
		\item $\Lambda_C$ is integral if and only if $C \subset C^{\perp}$.
		\item $\Lambda_C$ is unimodular if and only if $C$ is Euclidean self-dual.
		\item $\Lambda_C$ is even if and only if $\mathrm{wt}_{\mathrm{L}}(x) \in 4\Z$ for all $x \in C$.
		\item $\Lambda_C$ is even unimodular if and only if $C$ is Type II. 
	\end{enumerate}
	\item[\rm (II)] When $n \in 4\Z$ and $n \notin 8\Z$:
	\begin{enumerate}
		\item $\Lambda_C$ is integral if and only if $C \subset C^{\perp}$.
		\item $\Lambda_C$ is unimodular if and only if $C$ is Euclidean self-dual.
		\item $\Lambda_C$ is even if and only if $\mathrm{wt}_{\mathrm{H}}(x) \in 2\Z$ for all $x \in C$.
		\item $\Lambda_C$ is even unimodular if and only if $C$ is Type IV. 
	\end{enumerate}
	\item[\rm (III)] When $n \in 8\Z$:
	\begin{enumerate}
		\item $\Lambda_C$ is integral if and only if $C \subset C^{\perp}$.
		\item $\Lambda_C$ is unimodular if and only if $C$ is Euclidean self-dual.
		\item $\Lambda_C$ is even if and only if $\mathrm{wt}_{\mathrm{B}}(x) \in 2\Z$ for all $x \in C$.
		\item $\Lambda_C$ is even unimodular if and only if $C$ is Type IV. 
	\end{enumerate}
\end{enumerate}
\end{theorem}

\begin{remark}
Let $C$ be a Euclidean self-dual code over $\F_2[u]/(u^2)$. Then, $\mathrm{wt}_{\mathrm{H}}(x) \in 2\Z$ for all $x \in C$ if and only if $\mathrm{wt}_{\mathrm{B}}(x) \in 2\Z$ for all $x \in C$.
\end{remark}

%
%
%


\subsection{Lattices of type $D_n$ and $\F_4$-codes}\label{Dnsub3}
In this section, we construct lattices from codes over the finite field $\F_4 = \{0, 1, \omega, \omega^2\}$. The definition of the lattice $\Lambda$ of type $D_n$ for an even integer $n \geq 4$, the choice of the bases $(f_1, \dots, f_n)$ and $(f_1^*, \dots, f_n^*)$, and the corresponding inner products are completely analogous to those presented in Section \ref{Dnsub2}. 

Let $m$ be a positive integer. We define the mapping $\rho^{\oplus m} : (\Lambda^*)^{\oplus m} \longrightarrow \F_4^{\oplus m}$ by
\[
\rho^{\oplus m} \left( \left( \sum_{j=1}^{n} x_{ij} f^*_j \right)_{1\leq i \leq m} \right) := \left( \left( \sum_{j=1}^{n/2} x_{i(2j-1)} \bmod 2 \right) \omega + \left( \sum_{j=1}^{n/2} x_{i(2j)} \bmod 2 \right) \omega^2 \right)_{1\leq i \leq m},
\]
where $x_{ij} \in \Z$ for $1 \leq i \leq m$ and $1 \leq j \leq n$. It is clear that $\rho^{\oplus m}$ is an additive homomorphism.

Let $C$ be an $\F_4$-code of length $m$. We define the subset $\Lambda_C \subset (\Lambda^*)^{\oplus m}$ as the preimage of $C$:
\[
\Lambda_C := (\rho^{\oplus m})^{-1}(C).
\]

For $x = (x_1, \dots, x_m)$ and $y = (y_1, \dots, y_m)$ in $(\Lambda^*)^{\oplus m}$ with $x_i, y_i \in \Lambda^*$, we define a symmetric bilinear form $B$ on $(\Lambda^*)^{\oplus m}$ by
\[
B(x,y) := \sum_{i=1}^{m} b(x_i,y_i).
\]
Then, $\Lambda_C$ equipped with the bilinear form $B$ forms a lattice of rank $mn$.


We present several properties of the bilinear form $B$ and the lattice $\Lambda_C$.

\begin{lemma}\label{Dn3lem1}
Let $x,y \in (\Lambda^*)^{\oplus m}$.
\begin{enumerate}
    \item[\rm (I)] When $n \notin 4\Z$:
    \begin{align*}
    B(x,y) \in \Z &\iff \rho^{\oplus m}(x) \cdot \rho^{\oplus m}(y) \in \{ 0, 1 \}, \\
    B(x,x) \in 2\Z &\iff \mathrm{wt}_{\mathrm{L}}(\rho^{\oplus m}(x)) \in 4\Z.
    \end{align*}
    \item[\rm (II)] When $n \in 4\Z$:
    \begin{align*}
    B(x,y) \in \Z &\iff \rho^{\oplus m}(x) \cdot \overline{\rho^{\oplus m}(y)} \in \{0,1\}, \\
    B(x,x) \in 2\Z &\iff \begin{cases}
    \rho^{\oplus m}(x) \cdot \overline{\rho^{\oplus m}(x)} = 0 & \text{if } n \notin 8\Z, \\
    \mathrm{wt}_{\mathrm{B}}\left(\rho^{\oplus m}(x)\right) \in 2\Z & \text{if } n \in 8\Z.
    \end{cases}
    \end{align*}
\end{enumerate}
\end{lemma}

\begin{proof}
Let $x=(x_i)_{1\leq i \leq m}$ and $y=(y_i)_{1\leq i \leq m}$ be elements in $(\Lambda^*)^{\oplus m}$, where $x_i=\sum_{j=1}^{n}x_{ij}f_j^*$ and $y_i=\sum^{n}_{j=1} y_{ij}f_j^*$ with $x_{ij}, y_{ij}\in \Z$ for $1 \leq i \leq m$ and $1 \leq j \leq n$. Similar to \eqref{Eq:Dn} in Lemma \ref{Dn2lem1}, we have
\begin{align*}
	b(x_i, y_i)
	&= \sum_{j=1}^{n/2}\sum_{s=1}^{n/2}\left(\frac{n}{4}-|s-j|\right)x_{i(2j-1)}y_{i(2s-1)} + \sum_{j=1}^{n/2}\sum_{s=1}^{n/2}\left(\frac{n}{4}-|s-j|\right)x_{i(2j)}y_{i(2s)} \\
	&\quad + \sum_{j=1}^{n/2}\sum_{s=1}^{n/2}\left(\frac{n}{4}-\frac{1}{2}-\left|s-j\right|\right)x_{i(2j-1)}y_{i(2s)} + \sum_{j=1}^{n/2}\sum_{s=1}^{n/2}\left(\frac{n}{4}-\frac{1}{2}-|s-j|\right)x_{i(2j)}y_{i(2s-1)} \\
	&\quad + \sum_{j=1}^{(n/2)-1}(x_{i(2j+1)}y_{i(2j)}+x_{i(2j)}y_{i(2j+1)}).
\end{align*}

Let $n \notin 4\Z$. Since $n \in 2\Z$, we see that
\[
b(x_i,y_i) \in \frac{1}{2} \left( \left(\sum_{j=1}^{n/2}x_{i(2j-1)}\right) \left(\sum_{s=1}^{n/2}y_{i(2s-1)}\right) + \left(\sum_{j=1}^{n/2}x_{i(2j)}\right) \left(\sum_{s=1}^{n/2}y_{i(2s)}\right) \right) + \Z.  
\]
By considering the parity of $\sum_{j=1}^{n/2}x_{i(2j-1)}$, $\sum_{j=1}^{n/2}x_{i(2j)}$, $\sum_{j=1}^{n/2}y_{i(2j-1)}$, and $\sum_{j=1}^{n/2}y_{i(2j)}$, which corresponds to the pair $(\rho^{\oplus 1}(x_i), \rho^{\oplus 1}(y_i))$, we find that
\[
b(x_i,y_i) \in \begin{dcases}
	\Z & \text{if } \rho^{\oplus 1}(x_i) \cdot \rho^{\oplus 1}(y_i) \in \{ 0,1\}, \\
	\frac{1}{2}+\Z & \text{if } \rho^{\oplus 1}(x_i) \cdot \rho^{\oplus 1}(y_i) \in \{ \omega, \overline{\omega}\}.
\end{dcases}
\]
Hence,
\[
B(x,y) \in \Z \iff \rho^{\oplus m}(x) \cdot \rho^{\oplus m}(y) \in \{0,1\}.
\]

Let $n \in 4\Z$. We see that
\[
b(x_i,y_i) \in \frac{1}{2} \left( \left(\sum_{j=1}^{n/2}x_{i(2j-1)}\right) \left(\sum_{s=1}^{n/2}y_{i(2s)}\right) + \left(\sum_{j=1}^{n/2}x_{i(2j)}\right) \left(\sum_{s=1}^{n/2}y_{i(2s-1)}\right) \right) + \Z.  
\]
By considering the pair $(\rho^{\oplus 1}(x_i), \rho^{\oplus 1}(y_i))$, we have
\[
b(x_i,y_i) \in \begin{dcases}
	\Z & \text{if } \rho^{\oplus 1}(x_i) \cdot \overline{\rho^{\oplus 1}(y_i)} \in \{ 0,1\}, \\
	\frac{1}{2}+\Z & \text{if } \rho^{\oplus 1}(x_i) \cdot \overline{\rho^{\oplus 1}(y_i)} \in \{ \omega, \overline{\omega}\}.
\end{dcases}
\]
Hence,
\[
B(x,y) \in \Z \iff \rho^{\oplus m}(x) \cdot \overline{\rho^{\oplus m}(y)} \in \{0,1\}.
\]

Moreover, for the quadratic form, we have 
\begin{align*}
b(x_i,x_i) 
&= \frac{n}{4}\sum_{j=1}^{n/2}\sum_{s=1}^{n/2}x_{i(2j-1)}x_{i(2s-1)} + \frac{n}{4}\sum_{j=1}^{n/2}\sum_{s=1}^{n/2}x_{i(2j)}x_{i(2s)} \\
	&\quad - 2\sum_{1\leq j<s\leq n/2}|s-j|(x_{i(2j-1)}x_{i(2s-1)} + x_{i(2j)}x_{i(2s)}) \\
	&\quad + 2\sum_{j=1}^{n/2}\sum_{s=1}^{n/2}\left(\frac{n}{4}-\frac{1}{2}-\left|s-j\right|\right)x_{i(2j-1)}x_{i(2s)} + 2\sum_{j=1}^{(n/2)-1}x_{i(2j+1)}x_{i(2j)}.	
\end{align*}

Let $n \notin 4\Z$. Since $n \in 2\Z$, we see that
\[
b(x_i,x_i) \in \frac{n}{4} \left( \left(\sum_{j=1}^{n/2}x_{i(2j-1)}\right)^2 + \left(\sum_{j=1}^{n/2}x_{i(2j)}\right)^2 \right) + 2\Z. 
\]
Thus,
\[
2b(x_i,x_i) \equiv \begin{cases}
	0 \pmod 4 & \text{if } \rho^{\oplus 1}(x_i)=0, \\
	1 \pmod 4 & \text{if } \rho^{\oplus 1}(x_i)=\omega \text{ or } \overline{\omega}, \\
	2 \pmod 4 & \text{if } \rho^{\oplus 1}(x_i)=1. 
\end{cases}	
\]
Hence,
\[
B(x,x) \in 2\Z \iff \mathrm{wt}_{\mathrm{L}}(\rho^{\oplus m}(x)) \in 4\Z.
\]

Let $n \in 4\Z$ and $n \notin 8\Z$. We see that 
\[
b(x_i,x_i) \in \frac{n}{4} \left( \left(\sum_{j=1}^{n/2}x_{i(2j-1)}\right)^2 + \left(\sum_{j=1}^{n/2}x_{i(2j)}\right)^2 \right) - \left(\sum_{j=1}^{n/2}x_{i(2j-1)}\right)\left(\sum_{j=1}^{n/2}x_{i(2j)}\right) + 2\Z.
\]
Noting that $a \cdot \overline{a} \in \{0,1\}$ for all $a \in \F_4$, we obtain
\[
b(x_i,x_i) \equiv \begin{cases}
	0 \pmod 2 & \text{if } \rho^{\oplus 1}(x_i) \cdot \overline{\rho^{\oplus 1}(x_i)} = 0, \\
	1 \pmod 2 & \text{if } \rho^{\oplus 1}(x_i) \cdot \overline{\rho^{\oplus 1}(x_i)} = 1. 
\end{cases}	
\]
Hence,
\[
B(x,x) \in 2\Z \iff \rho^{\oplus m}(x) \cdot \overline{\rho^{\oplus m}(x)} = 0.
\]

Let $n \in 8\Z$. We see that 
\[
b(x_i,x_i) \in \left(\sum_{j=1}^{n/2}x_{i(2j-1)}\right)\left(\sum_{j=1}^{n/2}x_{i(2j)}\right) + 2\Z.
\]
Thus,
\[
b(x_i,x_i) \equiv \begin{cases}
	0 \pmod 2 & \text{if } \rho^{\oplus 1}(x_i) \in \{0, \omega, \overline{\omega}\}, \\
	1 \pmod 2 & \text{if } \rho^{\oplus 1}(x_i) = 1. 
\end{cases}	
\]
Hence,
\[
B(x,x) \in 2\Z \iff \mathrm{wt}_{\mathrm{B}}\left(\rho^{\oplus m}(x)\right) \in 2\Z.
\]
\end{proof}

\begin{lemma}\label{Dn3lem2}
Let $C$ be an $\F_4$-code of length $m$. We define
\[
C' := \{ x \in \F_4^{\oplus m} \mid x \cdot y \in \{0,1\} \text{ for all } y \in C \},
\]	
and
\[
C'' := \{ x \in \F_4^{\oplus m} \mid x \cdot \overline{y} \in \{0,1\} \text{ for all } y \in C \}.
\]
Then, $C'=C^{\perp}$ and $C''=\overline{C}^{\perp}$ hold.
\end{lemma}

\begin{proof}
The assertions can be proven in the exact same manner as Lemma \ref{Dn2lem2}.
\end{proof}

From Lemma \ref{Dn3lem2}, we can determine the dual lattice of $\Lambda_C$ in the same way as Lemma \ref{Dn2lem3}.

\begin{lemma}\label{Dn3lem3}
Let $C$ be an $\F_4$-code of length $m$. Then,
\[
\Lambda_C^* = \begin{cases}
\Lambda_{C^{\perp}}           & \text{if } n \in 2\Z \text{ and } n \notin 4\Z, \\
\Lambda_{\overline{C}^{\perp}} & \text{if } n \in 4\Z.
\end{cases}
\]
\end{lemma}

From Lemmas \ref{Dn3lem1} and \ref{Dn3lem3}, we obtain the following classification theorem.

\begin{theorem}\label{Dn3the1}
Let $C$ be an $\F_4$-code of length $m$.
\begin{enumerate}
	\item[\rm (I)] When $n \in 2\Z$ and $n \notin 4\Z$:
	\begin{enumerate}
		\item $\Lambda_C$ is integral if and only if $C \subset C^{\perp}$.
		\item $\Lambda_C$ is unimodular if and only if $C$ is Euclidean self-dual.
		\item $\Lambda_C$ is even if and only if $\mathrm{wt}_{\mathrm{L}}(x) \in 4\Z$ for all $x \in C$.
		\item $\Lambda_C$ is even unimodular if and only if $C$ is Type II. 
	\end{enumerate}
	\item[\rm (II)] When $n \in 4\Z$ and $n \notin 8\Z$:
	\begin{enumerate}
		\item $\Lambda_C$ is integral if and only if $C \subset \overline{C}^{\perp}$.
		\item $\Lambda_C$ is unimodular if and only if $C$ is Hermitian self-dual.
		\item $\Lambda_C$ is even if and only if $\mathrm{wt}_{\mathrm{H}}(x) \in 2\Z$ for all $x \in C$.
		\item $\Lambda_C$ is even unimodular if and only if $C$ is Hermitian self-dual. 
	\end{enumerate}
	\item[\rm (III)] When $n \in 8\Z$:
	\begin{enumerate}
		\item $\Lambda_C$ is integral if and only if $C \subset \overline{C}^{\perp}$.
		\item $\Lambda_C$ is unimodular if and only if $C$ is Hermitian self-dual.
		\item $\Lambda_C$ is even if and only if $\mathrm{wt}_{\mathrm{B}}(x) \in 2\Z$ for all $x \in C$.
		\item $\Lambda_C$ is even unimodular if and only if $C$ is Hermitian self-dual and $\mathrm{wt}_{\mathrm{B}}(x) \in 2\Z$ for all $x \in C$. 
	\end{enumerate}
\end{enumerate}
\end{theorem}

Note that for an $\F_4$-code $C$, we have $\mathrm{wt}_{\mathrm{H}}(x) \in 2\Z$ for all $x \in C$ if and only if $C \subset \overline{C}^{\perp}$ (see \cite[Theorem~1]{MACWILLIAMS1978288}).

\begin{remark}
Let $C$ be a Hermitian self-dual code of length $m$ over $\F_4$, and let $x \in C$. Since $x \cdot \overline{x} = 0$, we have
\[
(x + \overline{x}) \cdot (x + \overline{x}) + x \cdot \overline{x} = x \cdot x + \overline{x} \cdot \overline{x}.
\]
Consequently, it follows that
\[
\mathrm{wt}_{\mathrm{B}}(x) \in 2\Z \iff x \cdot x + \overline{x} \cdot \overline{x} = 0 \iff \mathrm{wt}_{\mathrm{L}}(x) \in 2\Z.
\]
Thus, $\Lambda_C$ is even if and only if $\mathrm{wt}_{\mathrm{L}}(x) \in 2\Z$ for all $x \in C$. 
\end{remark}

%
%
%

\subsection{Lattices of type $D_n$ and $(\F_2\times \F_2)$-codes}\label{Dnsub4}
In this section, we construct lattices from codes over the direct product ring $\F_2 \times \F_2$. The definition of the lattice $\Lambda$ of type $D_n$ for an even integer $n \geq 4$, the choice of the bases $(f_1, \dots, f_n)$ and $(f_1^*, \dots, f_n^*)$, and the corresponding inner products are completely analogous to those presented in Sections \ref{Dnsub2} and \ref{Dnsub3}.

Let $m$ be a positive integer. We define the mapping $\rho^{\oplus m} : (\Lambda^*)^{\oplus m} \longrightarrow (\F_2\times \F_2)^{\oplus m}$ by
\[
\rho^{\oplus m} \left( \left( \sum_{j=1}^{n} x_{ij} f^*_j \right)_{1\leq i \leq m} \right) := \left( \left( \sum_{j=1}^{n/2} x_{i(2j-1)} \bmod 2 \right), \left( \sum_{j=1}^{n/2} x_{i(2j)} \bmod 2 \right) \right)_{1\leq i \leq m},
\]
where $x_{ij} \in \Z$ for $1 \leq i \leq m$ and $1 \leq j \leq n$. It is clear that $\rho^{\oplus m}$ is an additive homomorphism.

Let $C$ be an $(\F_2\times \F_2)$-code of length $m$. We define the subset $\Lambda_C \subset (\Lambda^*)^{\oplus m}$ as the preimage of $C$:
\[
\Lambda_C := (\rho^{\oplus m})^{-1}(C).
\]

For $x = (x_1, \dots, x_m)$ and $y = (y_1, \dots, y_m)$ in $(\Lambda^*)^{\oplus m}$ with $x_i, y_i \in \Lambda^*$, we define a symmetric bilinear form $B$ on $(\Lambda^*)^{\oplus m}$ by
\[
B(x,y) := \sum_{i=1}^{m} b(x_i,y_i).
\]
Then, $\Lambda_C$ equipped with the bilinear form $B$ forms a lattice of rank $mn$.


We present several properties of the bilinear form $B$ and the lattice $\Lambda_C$.

\begin{lemma}\label{Dn4lem1}
Let $x,y \in (\Lambda^*)^{\oplus m}$.
\begin{enumerate}
    \item[\rm (I)] When $n \notin 4\Z$:
    \begin{align*}
    B(x,y) \in \Z &\iff \rho^{\oplus m}(x) \cdot \rho^{\oplus m}(y) \in \{ (0,0), (1,1) \}, \\
    B(x,x) \in 2\Z &\iff \mathrm{wt}_{\mathrm{L}}(\rho^{\oplus m}(x)) \in 4\Z.
    \end{align*}
    \item[\rm (II)] When $n \in 4\Z$:
    \begin{align*}
    B(x,y) \in \Z &\iff \rho^{\oplus m}(x) \cdot \overline{\rho^{\oplus m}(y)} \in \{ (0,0), (1,1) \}, \\
    B(x,x) \in 2\Z &\iff \begin{cases}
    \mathrm{wt}_{\mathrm{H}}\left(\rho^{\oplus m}(x)\right) \in 2\Z & \text{if } n \notin 8\Z, \\
    \rho^{\oplus m}(x) \cdot \overline{\rho^{\oplus m}(x)} = (0,0) & \text{if } n \in 8\Z.
    \end{cases}
    \end{align*}
\end{enumerate}
\end{lemma}

\begin{proof}
Let $x=(x_i)_{1\leq i \leq m}$ and $y=(y_i)_{1\leq i \leq m}$ be elements in $(\Lambda^*)^{\oplus m}$, where $x_i=\sum_{j=1}^{n}x_{ij}f_j^*$ and $y_i=\sum^{n}_{j=1} y_{ij}f_j^*$ with $x_{ij}, y_{ij}\in \Z$ for $1 \leq i \leq m$ and $1 \leq j \leq n$. Similar to \eqref{Eq:Dn} in Lemma \ref{Dn2lem1}, we have
\begin{align*}
	b(x_i, y_i)
	&= \sum_{j=1}^{n/2}\sum_{s=1}^{n/2}\left(\frac{n}{4}-|s-j|\right)x_{i(2j-1)}y_{i(2s-1)} + \sum_{j=1}^{n/2}\sum_{s=1}^{n/2}\left(\frac{n}{4}-|s-j|\right)x_{i(2j)}y_{i(2s)} \\
	&\quad + \sum_{j=1}^{n/2}\sum_{s=1}^{n/2}\left(\frac{n}{4}-\frac{1}{2}-\left|s-j\right|\right)x_{i(2j-1)}y_{i(2s)} + \sum_{j=1}^{n/2}\sum_{s=1}^{n/2}\left(\frac{n}{4}-\frac{1}{2}-|s-j|\right)x_{i(2j)}y_{i(2s-1)} \\
	&\quad + \sum_{j=1}^{(n/2)-1}(x_{i(2j+1)}y_{i(2j)}+x_{i(2j)}y_{i(2j+1)}).
\end{align*}

Let $n \notin 4\Z$. Since $n \in 2\Z$, we see that
\begin{align*}
b(x_i,y_i) \in \Z 
&\iff \frac{n}{4} \left( \left(\sum_{j=1}^{n/2}x_{i(2j-1)}\right) \left(\sum_{s=1}^{n/2}y_{i(2s-1)}\right) + \left(\sum_{j=1}^{n/2}x_{i(2j)}\right) \left(\sum_{s=1}^{n/2}y_{i(2s)}\right) \right) \in \Z \\ 
&\iff \rho^{\oplus 1}(x_i) \cdot \rho^{\oplus 1}(y_i) \in \{ (0,0), (1,1) \}.
\end{align*}
Hence,
\[
B(x,y) \in \Z \iff \rho^{\oplus m}(x) \cdot \rho^{\oplus m}(y) \in \{(0,0), (1,1)\}.
\]

Let $n \in 4\Z$. We see that
\begin{align*}
b(x_i,y_i) \in \Z 
&\iff -\frac{1}{2} \left( \left(\sum_{j=1}^{n/2}x_{i(2j-1)}\right) \left(\sum_{s=1}^{n/2}y_{i(2s)}\right) + \left(\sum_{j=1}^{n/2}x_{i(2j)}\right) \left(\sum_{s=1}^{n/2}y_{i(2s-1)}\right) \right) \in \Z \\ 
&\iff \rho^{\oplus 1}(x_i) \cdot \overline{\rho^{\oplus 1}(y_i)} \in \{ (0,0), (1,1) \}.
\end{align*}
Hence,
\[
B(x,y) \in \Z \iff \rho^{\oplus m}(x) \cdot \overline{\rho^{\oplus m}(y)} \in \{(0,0), (1,1)\}.
\]

Moreover, for the quadratic form, we have 
\begin{align*}
b(x_i,x_i) 
&= \frac{n}{4}\sum_{j=1}^{n/2}\sum_{s=1}^{n/2}x_{i(2j-1)}x_{i(2s-1)} + \frac{n}{4}\sum_{j=1}^{n/2}\sum_{s=1}^{n/2}x_{i(2j)}x_{i(2s)} \\
	&\quad - 2\sum_{1\leq j<s\leq n/2}|s-j|(x_{i(2j-1)}x_{i(2s-1)} + x_{i(2j)}x_{i(2s)}) \\
	&\quad + 2\sum_{j=1}^{n/2}\sum_{s=1}^{n/2}\left(\frac{n}{4}-\frac{1}{2}-\left|s-j\right|\right)x_{i(2j-1)}x_{i(2s)} + 2\sum_{j=1}^{(n/2)-1}x_{i(2j+1)}x_{i(2j)}.	
\end{align*}

Let $n \notin 4\Z$. Since $n \in 2\Z$, we see that
\[
b(x_i,x_i) \in \frac{n}{4} \left( \left(\sum_{j=1}^{n/2}x_{i(2j-1)}\right)^2 + \left(\sum_{j=1}^{n/2}x_{i(2j)}\right)^2 \right) + 2\Z. 
\]
Thus,
\[
2b(x_i,x_i) \equiv \begin{cases}
	0 \pmod 4 & \text{if } \rho^{\oplus 1}(x_i)=(0,0), \\
	1 \pmod 4 & \text{if } \rho^{\oplus 1}(x_i)=(1,0) \text{ or } (0,1), \\
	2 \pmod 4 & \text{if } \rho^{\oplus 1}(x_i)=(1,1). 
\end{cases}	
\]
Hence,
\[
B(x,x) \in 2\Z \iff \mathrm{wt}_{\mathrm{L}}(\rho^{\oplus m}(x)) \in 4\Z.
\]

Let $n \in 4\Z$ and $n \notin 8\Z$. We see that 
\[
b(x_i,x_i) \in \frac{n}{4} \left( \left(\sum_{j=1}^{n/2}x_{i(2j-1)}\right)^2 + \left(\sum_{j=1}^{n/2}x_{i(2j)}\right)^2 \right) - \left(\sum_{j=1}^{n/2}x_{i(2j-1)}\right)\left(\sum_{j=1}^{n/2}x_{i(2j)}\right) + 2\Z.
\]
Thus,
\[
b(x_i,x_i) \equiv \begin{cases}
	0 \pmod 2 & \text{if } \rho^{\oplus 1}(x_i)=(0,0), \\
	1 \pmod 2 & \text{otherwise}. 
\end{cases}	
\]
Hence,
\[
B(x,x) \in 2\Z \iff \mathrm{wt}_{\mathrm{H}}\left(\rho^{\oplus m}(x)\right) \in 2\Z.
\]

Let $n \in 8\Z$. We see that 
\[
b(x_i,x_i) \in \left(\sum_{j=1}^{n/2}x_{i(2j-1)}\right)\left(\sum_{j=1}^{n/2}x_{i(2j)}\right) + 2\Z.
\]
Thus,
\[
b(x_i,x_i) \in 2\Z \iff \rho^{\oplus 1}(x_i) \cdot \overline{\rho^{\oplus 1}(x_i)} = (0,0).
\]
Hence,
\[
B(x,x) \in 2\Z \iff \rho^{\oplus m}(x) \cdot \overline{\rho^{\oplus m}(x)} = (0,0).
\]
\end{proof}


\begin{lemma}\label{Dn4lem2}
Let $C$ be an $(\F_2\times \F_2)$-code of length $m$. We define
\[
C' := \{ x \in (\F_2\times \F_2)^{\oplus m} \mid x \cdot y \in \{(0,0),(1,1)\} \text{ for all } y \in C \},
\]	
and
\[
C'' := \{ x \in (\F_2\times \F_2)^{\oplus m} \mid x \cdot \overline{y} \in \{(0,0),(1,1)\} \text{ for all } y \in C \}.
\]
Then, $C'=C^{\perp}$ and $C''=\overline{C}^{\perp}$ hold.
\end{lemma}

\begin{proof}
The assertions can be proven in the exact same manner as Lemma \ref{Dn2lem2}.
\end{proof}

From Lemma \ref{Dn4lem2}, we can determine the dual lattice of $\Lambda_C$ in the same way as Lemma \ref{Dn2lem3}.

\begin{lemma}\label{Dn4lem3}
Let $C$ be an $(\F_2\times \F_2)$-code of length $m$. Then,
\[
\Lambda_C^* = \begin{cases}
\Lambda_{C^{\perp}}           & \text{if } n \in 2\Z \text{ and } n \notin 4\Z, \\
\Lambda_{\overline{C}^{\perp}} & \text{if } n \in 4\Z.
\end{cases}
\]
\end{lemma}

From Lemmas \ref{Dn4lem1} and \ref{Dn4lem3}, we obtain the following classification theorem.

\begin{theorem}\label{Dn4the1}
Let $C$ be an $(\F_2\times \F_2)$-code of length $m$.
\begin{enumerate}
	\item[\rm (I)] When $n \in 2\Z$ and $n \notin 4\Z$:
	\begin{enumerate}
		\item $\Lambda_C$ is integral if and only if $C \subset C^{\perp}$.
		\item $\Lambda_C$ is unimodular if and only if $C$ is Euclidean self-dual.
		\item $\Lambda_C$ is even if and only if $\mathrm{wt}_{\mathrm{L}}(x) \in 4\Z$ for all $x \in C$.
		\item $\Lambda_C$ is even unimodular if and only if $C$ is Type II. 
	\end{enumerate}
	\item[\rm (II)] When $n \in 4\Z$ and $n \notin 8\Z$:
	\begin{enumerate}
		\item $\Lambda_C$ is integral if and only if $C \subset \overline{C}^{\perp}$.
		\item $\Lambda_C$ is unimodular if and only if $C$ is Hermitian self-dual.
		\item $\Lambda_C$ is even if and only if $\mathrm{wt}_{\mathrm{H}}(x) \in 2\Z$ for all $x \in C$.
		\item $\Lambda_C$ is even unimodular if and only if $C$ is Type IV. 
	\end{enumerate}
	\item[\rm (III)] When $n \in 8\Z$:
	\begin{enumerate}
		\item $\Lambda_C$ is integral if and only if $C \subset \overline{C}^{\perp}$.
		\item $\Lambda_C$ is unimodular if and only if $C$ is Hermitian self-dual.
		\item $\Lambda_C$ is even if and only if $C \subset \overline{C}^{\perp}$.
		\item $\Lambda_C$ is even unimodular if and only if $C$ is Hermitian self-dual. 
	\end{enumerate}
\end{enumerate}
\end{theorem}

\subsection{Lattices of type $E_6$ and $\Z/3\Z$-codes}\label{E6sub}
Let $(\Lambda,b)$ be a lattice of type $E_6$. Then, $\Lambda$ has a basis $(e_1, \dots ,e_6)$ such that its Gram matrix is given by
\[
(b( e_{i}, e_{j} ))_{1\leq i,j \leq 6} = 
\begin{pmatrix}
    2 & -1 &    &    &    &    \\
   -1 &  2 & -1 &    &    &    \\
      & -1 &  2 & -1 &    & -1 \\
      &    & -1 &  2 & -1 &    \\
      &    &    & -1 &  2 &    \\ 
      &    & -1 &    &    &  2
\end{pmatrix},
\]
where the blank entries denote $0$.

We define a new basis $(f_1, \dots, f_6)$ by
\[
(f_1, \dots, f_6) := (e_1, \dots, e_6)
\begin{pmatrix}
      &    & 1 &   &   &   \\
    1 &    &   &   & 1 & 1 \\
    1 &  1 &   &   & 1 & 1 \\
      &    &   & 1 & 1 &   \\
      &    &   &   & 1 &   \\ 
      &    &   &   &   & 1
\end{pmatrix},
\]
and we define the corresponding dual elements $(f^*_1, \dots, f^*_6)$ by
\[
(f^*_1, \dots, f^*_6) := (f_1, \dots, f_6) \cdot \frac{1}{3}
\begin{pmatrix}
    4 &  1 & 1 & 2 & -1 & -1 \\
    1 &  4 & 1 & 2 &  2 & -1 \\
    1 &  1 & 4 & 2 &  2 &  2 \\
    2 &  2 & 2 & 4 &  1 &  1 \\
   -1 &  2 & 2 & 1 &  4 &  1 \\ 
   -1 & -1 & 2 & 1 &  1 &  4
\end{pmatrix},
\]
where the product of the row vector and the matrix is defined in the standard manner, i.e., $(a_1, \dots, a_n) (b_{ij})_{1 \leq i,j \leq n} = \left( \sum_{k=1}^n a_k b_{k1}, \dots, \sum_{k=1}^n a_k b_{kn} \right)$.

By direct calculation, the inner products among these vectors are given by
\[
(b( f_{i}, f_{j} ))_{1\leq i,j \leq 6} =
\begin{pmatrix}
    2 &  0 & -1 & -1 &  1 &  1 \\
    0 &  2 &  0 & -1 & -1 &  1 \\
   -1 &  0 &  2 &  0 & -1 & -1 \\
   -1 & -1 &  0 &  2 &  0 & -1 \\
    1 & -1 & -1 &  0 &  2 &  0 \\ 
    1 &  1 & -1 & -1 &  0 &  2
\end{pmatrix},
\]
and
\[
(b( f^*_{i}, f^*_{j} ))_{1\leq i,j \leq 6} = 
\frac{1}{3}
\begin{pmatrix}
    4 &  1 & 1 & 2 & -1 & -1 \\
    1 &  4 & 1 & 2 &  2 & -1 \\
    1 &  1 & 4 & 2 &  2 &  2 \\
    2 &  2 & 2 & 4 &  1 &  1 \\
   -1 &  2 & 2 & 1 &  4 &  1 \\ 
   -1 & -1 & 2 & 1 &  1 &  4
\end{pmatrix}.
\]
Thus, we obtain
\[
b(f_i, f_j^*) = \delta_{ij},
\]
where $\delta_{ij}$ denotes the Kronecker delta. Hence, $(f_1, \dots, f_6)$ forms a basis of $\Lambda$, and $(f_1^*, \dots, f_6^*)$ forms a basis of the dual lattice $\Lambda^*$.


We consider the mapping $\rho^{\oplus m} : (\Lambda^*)^{\oplus m} \longrightarrow (\Z/3\Z)^{\oplus m}$ defined by
\[
\rho^{\oplus m} \left( \left( \sum_{j=1}^{6} x_{ij} f^*_j \right)_{1\leq i \leq m} \right) := \left( \left( \sum_{j=1}^{3} x_{ij} - \sum_{j=4}^{6} x_{ij} \right) \bmod 3 \right)_{1\leq i \leq m},
\]
where $x_{ij} \in \Z$ for $1 \leq i \leq m$ and $1 \leq j \leq 6$. It is clear that $\rho^{\oplus m}$ is an additive homomorphism.

Let $C$ be a $\Z/3\Z$-code of length $m$. We define the subset $\Lambda_C \subset (\Lambda^*)^{\oplus m}$ as the preimage of $C$:
\[
\Lambda_C := (\rho^{\oplus m})^{-1}(C).
\]

For $x = (x_1, \dots, x_m)$ and $y = (y_1, \dots, y_m)$ in $(\Lambda^*)^{\oplus m}$ with $x_i, y_i \in \Lambda^*$, we define a symmetric bilinear form $B$ on $(\Lambda^*)^{\oplus m}$ by
\[
B(x,y) := \sum_{i=1}^{m} b(x_i,y_i).
\]
Then, $\Lambda_C$ equipped with the bilinear form $B$ forms a lattice of rank $6m$.

We present several properties of the bilinear form $B$ and the lattice $\Lambda_C$.

\begin{lemma}\label{E6lem1}
Let $x,y \in (\Lambda^*)^{\oplus m}$. Then,
\begin{align*}
B(x,y) \in \Z &\iff \rho^{\oplus m}(x) \cdot \rho^{\oplus m}(y) = 0, \\
B(x,x) \in 2\Z &\iff \rho^{\oplus m}(x) \cdot \rho^{\oplus m}(x) = 0.
\end{align*}
\end{lemma}

\begin{proof}	
Let $x=(x_i)_{1\leq i \leq m}$ and $y=(y_i)_{1\leq i \leq m}$ be elements in $(\Lambda^*)^{\oplus m}$, where $x_i=\sum_{j=1}^{6}x_{ij}f_j^*$ and $y_i=\sum^{6}_{j=1} y_{ij}f_j^*$ with $x_{ij}, y_{ij}\in \Z$ for $1 \leq i \leq m$ and $1 \leq j \leq 6$. Then, we have
\begin{align*}
b(x_i,y_i)
&= \frac{1}{3}\left(\sum_{j=1}^{3}x_{ij}-\sum_{j=4}^{6}x_{ij}\right)\left(\sum_{j=1}^{3}y_{ij}-\sum_{j=4}^{6}y_{ij}\right) + \sum_{j=1}^{6}x_{ij}y_{ij} \\
&\quad + x_{i1}y_{i4} + x_{i4}y_{i1} + x_{i2}y_{i4} + x_{i4}y_{i2} + x_{i2}y_{i5} + x_{i5}y_{i2} \\
&\quad + x_{i3}y_{i4} + x_{i4}y_{i3} + x_{i3}y_{i5} + x_{i5}y_{i3} + x_{i3}y_{i6} + x_{i6}y_{i3}.
\end{align*}
Thus, it follows that
\[
b(x_i,y_i) \in \frac{1}{3}\left(\sum_{j=1}^{3}x_{ij}-\sum_{j=4}^{6}x_{ij}\right)\left(\sum_{j=1}^{3}y_{ij}-\sum_{j=4}^{6}y_{ij}\right) + \Z.
\]
Hence, we obtain
\[
B(x,y) \in \Z \iff \rho^{\oplus m}(x) \cdot \rho^{\oplus m}(y) = 0.
\]

Moreover, for the quadratic form, we have 
\begin{align*}
b(x_i,x_i)
&= \frac{4}{3}\left(\sum_{j=1}^{3}x_{ij}-\sum_{j=4}^{6}x_{ij}\right)^2 \\
&\quad + 2(x_{i1}x_{i2} + x_{i2}x_{i3} + x_{i3}x_{i1} + x_{i4}x_{i5} + x_{i5}x_{i6} + x_{i6}x_{i4} + x_{i1}x_{i5} + x_{i1}x_{i6} + x_{i2}x_{i6}).
\end{align*} 
Therefore, we conclude that
\[
B(x,x) \in 2\Z \iff \rho^{\oplus m}(x) \cdot \rho^{\oplus m}(x) = 0.
\]
\end{proof}

\begin{lemma}\label{E6lem2}
Let $C$ be a $\Z/3\Z$-code of length $m$. Then, the dual lattice of $\Lambda_C$ is given by
\[
\Lambda_C^* = \Lambda_{C^{\perp}}.
\]	
\end{lemma}

\begin{proof}
This assertion can be proven in the exact same manner as Lemma \ref{Anlem2}.
\end{proof}

From Lemmas \ref{E6lem1} and \ref{E6lem2}, we establish the following classification theorem.

\begin{theorem}\label{E6the}
Let $C$ be a $\Z/3\Z$-code of length $m$.
\begin{enumerate}
    \item $\Lambda_{C}$ is integral if and only if $C \subset C^{\perp}$.
    \item $\Lambda_{C}$ is unimodular if and only if $C$ is Euclidean self-dual. 
    \item $\Lambda_{C}$ is even if and only if $C \subset C^{\perp}$.
    \item $\Lambda_{C}$ is even unimodular if and only if $C$ is Euclidean self-dual.
\end{enumerate}	
\end{theorem}


\subsection{Lattices of type $E_7$ and $\Z/2\Z$-codes}\label{E7sub}
Let $(\Lambda,b)$ be a lattice of type $E_7$. Then, $\Lambda$ has a basis $(e_1, \dots ,e_7)$ such that its Gram matrix is given by
\[
(b( e_{i}, e_{j} ))_{1\leq i,j \leq 7} = 
\begin{pmatrix}
    2 & -1 &    &    &    &    &    \\
   -1 &  2 & -1 &    &    &    &    \\
      & -1 &  2 & -1 &    &    &    \\
      &    & -1 &  2 & -1 &    & -1 \\
      &    &    & -1 &  2 & -1 &    \\ 
      &    &    &    & -1 &  2 &    \\
      &    &    & -1 &    &    &  2 
\end{pmatrix},
\]
where the blank entries denote $0$.

We define a new basis $(f_1, \dots, f_7)$ by
\[
(f_1, \dots, f_7) := (e_1, \dots, e_7)
\begin{pmatrix}
    1 &   &   &   &   &   &    \\
    1 & 1 &   &   &   &   &    \\
    1 & 1 & 1 &   &   &   &    \\
    1 & 1 & 1 & 1 &   &   &    \\
    1 & 1 & 1 & 1 & 1 &   & -1 \\ 
    1 & 1 & 1 & 1 & 1 & 1 & -2 \\
      &   &   &   &   &   &  1 
\end{pmatrix},
\]
and we define the corresponding dual elements $(f^*_1, \dots, f^*_7)$ by
\[
(f^*_1, \dots, f^*_7) := (f_1, \dots, f_7) \cdot \frac{1}{2}
\begin{pmatrix}
    3 & 1 & 1 & 1 & 1 & 1 & 3 \\
    1 & 3 & 1 & 1 & 1 & 1 & 3 \\
    1 & 1 & 3 & 1 & 1 & 1 & 3 \\
    1 & 1 & 1 & 3 & 1 & 1 & 3 \\
    1 & 1 & 1 & 1 & 3 & 1 & 3 \\ 
    1 & 1 & 1 & 1 & 1 & 3 & 3 \\
    3 & 3 & 3 & 3 & 3 & 3 & 7 
\end{pmatrix}.
\]

By direct calculation, the inner products among these vectors are given by
\[
(b( f_{i}, f_{j} ))_{1\leq i,j \leq 7} = 
\begin{pmatrix}
    2 &  1 &  1 &  1 &  1 &  1 & -3 \\
    1 &  2 &  1 &  1 &  1 &  1 & -3 \\
    1 &  1 &  2 &  1 &  1 &  1 & -3 \\
    1 &  1 &  1 &  2 &  1 &  1 & -3 \\
    1 &  1 &  1 &  1 &  2 &  1 & -3 \\ 
    1 &  1 &  1 &  1 &  1 &  2 & -3 \\
   -3 & -3 & -3 & -3 & -3 & -3 &  8 
\end{pmatrix},
\]
and
\[
(b( f^*_{i}, f^*_{j} ))_{1\leq i,j \leq 7} = 
\frac{1}{2}
\begin{pmatrix}
    3 & 1 & 1 & 1 & 1 & 1 & 3 \\
    1 & 3 & 1 & 1 & 1 & 1 & 3 \\
    1 & 1 & 3 & 1 & 1 & 1 & 3 \\
    1 & 1 & 1 & 3 & 1 & 1 & 3 \\
    1 & 1 & 1 & 1 & 3 & 1 & 3 \\ 
    1 & 1 & 1 & 1 & 1 & 3 & 3 \\
    3 & 3 & 3 & 3 & 3 & 3 & 7 
\end{pmatrix}.
\]
Thus, we obtain
\[
b(f_i, f_j^*) = \delta_{ij},
\]
where $\delta_{ij}$ denotes the Kronecker delta. Hence, $(f_1, \dots, f_7)$ forms a basis of $\Lambda$, and $(f_1^*, \dots, f_7^*)$ forms a basis of the dual lattice $\Lambda^*$.  

We consider the mapping $\rho^{\oplus m} : (\Lambda^*)^{\oplus m} \longrightarrow (\Z/2\Z)^{\oplus m}$ defined by
\[
\rho^{\oplus m} \left( \left( \sum_{j=1}^{7} x_{ij} f^*_j \right)_{1\leq i \leq m} \right) := \left( \left( \sum_{j=1}^{7} x_{ij} \right) \bmod 2 \right)_{1\leq i \leq m},
\]
where $x_{ij} \in \Z$ for $1 \leq i \leq m$ and $1 \leq j \leq 7$. It is clear that $\rho^{\oplus m}$ is an additive homomorphism.

Let $C$ be a $\Z/2\Z$-code of length $m$. We define the subset $\Lambda_C \subset (\Lambda^*)^{\oplus m}$ as the preimage of $C$:
\[
\Lambda_C := (\rho^{\oplus m})^{-1}(C).
\]

For $x = (x_1, \dots, x_m)$ and $y = (y_1, \dots, y_m)$ in $(\Lambda^*)^{\oplus m}$ with $x_i, y_i \in \Lambda^*$, we define a symmetric bilinear form $B$ on $(\Lambda^*)^{\oplus m}$ by
\[
B(x,y) := \sum_{i=1}^{m} b(x_i,y_i).
\]
Then, $\Lambda_C$ equipped with the bilinear form $B$ forms a lattice of rank $7m$.


We present several properties of the bilinear form $B$ and the lattice $\Lambda_C$.

\begin{lemma}\label{E7lem1}
Let $x,y \in (\Lambda^*)^{\oplus m}$. Then,
\begin{align*}
B(x,y) \in \Z &\iff \rho^{\oplus m}(x) \cdot \rho^{\oplus m}(y) = 0, \\
B(x,x) \in 2\Z &\iff \mathrm{wt}_{\mathrm{E}}(\rho^{\oplus m}(x)) \in 4\Z.
\end{align*}
\end{lemma}

\begin{proof}	
Let $x=(x_i)_{1\leq i \leq m}$ and $y=(y_i)_{1\leq i \leq m}$ be elements in $(\Lambda^*)^{\oplus m}$, where $x_i=\sum_{j=1}^{7}x_{ij}f_j^*$ and $y_i=\sum^{7}_{j=1} y_{ij}f_j^*$ with $x_{ij}, y_{ij}\in \Z$ for $1 \leq i \leq m$ and $1 \leq j \leq 7$. By direct expansion, we have
\begin{align*}
b(x_i,y_i) = \frac{1}{2}\left(\sum_{j=1}^{7}x_{ij}\right)\left(\sum_{j=1}^{7}y_{ij}\right) + \sum_{j=1}^{7}x_{ij}y_{ij} + \sum_{j=1}^{7}(x_{ij}y_{i7} + x_{i7}y_{ij}).
\end{align*}
Thus, it follows that
\[
B(x,y) \in \Z \iff \rho^{\oplus m}(x) \cdot \rho^{\oplus m}(y) = 0. 
\]

Moreover, for the quadratic form, we obtain
\begin{align*}
b(x_i,x_i) = \frac{3}{2}\left(\sum_{j=1}^{7}x_{ij}\right)^2 - 2\sum_{1\leq j<k\leq 7 }x_{ij}x_{ik} + 2\sum_{j=1}^{7}x_{ij}x_{i7}.
\end{align*}
Therefore, we conclude that
\[
B(x,x) \in 2\Z \iff \mathrm{wt}_{\mathrm{E}}(\rho^{\oplus m}(x)) \in 4\Z.
\]
\end{proof}

\begin{lemma}\label{E7lem2}
Let $C$ be a $\Z/2\Z$-code of length $m$. Then, the dual lattice of $\Lambda_C$ is given by
\[
\Lambda_C^* = \Lambda_{C^{\perp}}.
\]	
\end{lemma}

\begin{proof}
This assertion can be proven in the exact same manner as Lemma \ref{Anlem2}.
\end{proof}

From Lemmas \ref{E7lem1} and \ref{E7lem2}, we obtain the following classification theorem.

\begin{theorem}\label{E7the}
Let $C$ be a $\Z/2\Z$-code of length $m$.
\begin{enumerate}
    \item $\Lambda_{C}$ is integral if and only if $C \subset C^{\perp}$.
    \item $\Lambda_{C}$ is unimodular if and only if $C$ is Euclidean self-dual. 
    \item $\Lambda_{C}$ is even if and only if $\mathrm{wt}_{\mathrm{E}}(x) \in 4\Z$ for all $x \in C$.
    \item $\Lambda_{C}$ is even unimodular if and only if $C$ is Type II.
\end{enumerate}	
\end{theorem}

\section{Applications to Hilbert modular forms}\label{Sect4}
In this section, we provide several applications of the lattices constructed in Section \ref{Sect3} to Hilbert modular forms. Van der Geer and Hirzebruch established the connections between the weight enumerators of $\F_p$-codes and the theta functions of lattices over number fields, where $p$ is an odd prime number (see \cite[Ch.~5]{Ebeling}). This result is understood to be deduced from the fact that a lattice of type $A_{p-1}$ exists over $\Q(\zeta_p)$ as a fractional ideal. Furthermore, it was shown in \cite{BayerFluckiger1999LatticesAN} that lattices of type $D_4$ and $E_6$ exist over certain number fields as fractional ideals. By tracing the methods and results in \cite[Ch.~5]{Ebeling}, we can extend these connections between the weight enumerators of $\F_p$-codes and the theta functions of lattices of type $A_{p-1}$ to the cases of lattices of type $D_4$ and $E_6$.


\subsection{Lattice of type $D_4$ over integers of $\Q(\zeta_8)$}\label{zeta8sub}
By Proposition \ref{K-latthe}, the theta function of an even lattice $\Lambda$ is a Hilbert modular form whenever $\Lambda$ is a $\Z_K$-lattice. In this section, after confirming that a specific lattice of type $D_4$ can be viewed as a $\Z_K$-lattice, we establish the connections between the weight enumerators of $\F_2[u]/(u^2)$-codes and the theta functions of lattices over $\Q(\zeta_8)$.

Let $F := \Q(\zeta_8)$ and define the subset $\Lambda \subset F$ by
\[
\Lambda := \frac{1-\zeta_8}{2}\Z_F.
\]
Then, $\Lambda$ has a $\Z$-basis 
\[
(e_1, e_2, e_3, e_4) = \left( \frac{1-\zeta_8}{2}, \frac{(1-\zeta_8)\zeta_8}{2}, \frac{(1-\zeta_8)\zeta_8^2}{2}, \frac{(1-\zeta_8)(\zeta_8^3-1-\zeta_8)}{2} \right),
\]
and its Gram matrix with respect to the trace form is given by
\[
\left( \Tr_{F/\Q}(e_i \overline{e_j}) \right)_{1\leq i,j\leq 4} =
\begin{pmatrix}
	2  & -1 &    &    \\
	-1 &  2 & -1 & -1 \\
	   & -1 &  2 &    \\
	   & -1 &    &  2  
\end{pmatrix}.
\]
Thus, $\left(\Lambda, \Tr_{F/\Q}(x\overline{y})\right)$ forms a lattice of type $D_4$.

Moreover, $\Lambda$ also possesses a $\Z$-basis
\[
f_i = \frac{1-\zeta_8}{2}\zeta_8^{i-1} \quad (1\leq i\leq 4),
\]
and its dual lattice $\Lambda^* = \frac{1}{2(1-\zeta_8)}\Z_F$ has a $\Z$-basis
\[
f_i^* = \sum_{l=1}^{4}\frac{1}{4}(4-2|i-l|)f_l = \frac{1}{2(1-\zeta_8^{-1})}\zeta_8^{i-1} \quad (1\leq i\leq 4).
\]

Let $C$ be an $\F_2[u]/(u^2)$-code of length $n$. By Theorem \ref{Dn2the1}, there exists a mapping $\rho^{\oplus n} : (\Lambda^*)^{\oplus n} \longrightarrow (\F_2[u]/(u^2))^{\oplus n}$ such that the lifted lattice $\Lambda_C := (\rho^{\oplus n})^{-1}(C)$, equipped with the trace form $\Tr_{F/\mathbb{Q}}(x\cdot \overline{y})$, is even unimodular if and only if $C$ is Type IV.

Let $K := \Q(\zeta_8+\zeta_8^{-1})$, $V := \Q(\zeta_8)^{\oplus n}$, and $\Gal(K/\Q) = \{\sigma_1, \sigma_2\}$. We define a totally positive definite scalar product on $F^{\oplus m}$ by
\[
\langle x, y \rangle_K := \Tr_{F/K}(x\cdot \overline{y}),
\]where $m$ is a positive integer.
Then, $\left(\Lambda, \langle \cdot,\cdot \rangle_K\right)$ and $\left(\Lambda_{C}, \langle \cdot, \cdot \rangle_K\right)$ can be naturally viewed as $\Z_K$-lattices.

From Proposition \ref{K-latthe}, we immediately obtain the following corollary.

\begin{corollary}\label{zeta8cor}
Let $C$ be a Type IV code over $(\F_2[u]/(u^2))$ of length $n$. Then, the theta function $\theta_{\Lambda_C}$ of the $\Z_K$-lattice $\Lambda_C$ is a Hilbert modular form of weight $n$ for the whole group $\mathrm{SL}_2(\Z_K)$.
\end{corollary}


Let $z = (z_1, z_2) \in \HH^{\oplus 2}$. For each $a \in \F_2[u]/(u^2)$, we define the theta function $\theta_a(z)$ by
\[
\theta_a(z) := \sum_{x \in x_a + \Lambda} \exp\left(\pi \sqrt{-1} \Tr_{K/\Q}(\langle x,x \rangle_K z)\right),
\]
where the coset representatives are chosen as $x_0 = 0$, $x_1 = f_1^*$, $x_{1+u} = f_2^*$, and $x_u = f_1^* + f_2^*$.

Since $f_2^* + \Lambda = \zeta_8(f_1^* + \Lambda)$, we have the equality of the sets of norms:
\begin{align*}
\{\langle x,x \rangle_K \mid x \in x_1 + \Lambda\} = \{\langle x,x \rangle_K \mid x \in x_{1+u} + \Lambda\}.
\end{align*}
Hence, it immediately follows that $\theta_1 = \theta_{1+u}$.

The weight enumerator of a code $C$ of length $n$ over $\F_2[u]/(u^2)$ is defined as the polynomial
\[
W_C(X_0, X_1, X_2) := \sum_{x \in C} X_0^{N_0(x)} X_1^{N_1(x)} X_2^{N_2(x)},
\]
where $(N_0(x), N_1(x), N_2(x))$ denotes the Lee composition of the codeword $x$.

We can now formulate the following main theorem of this section.

\begin{theorem}\label{zeta8thm2}
Let $C$ be an $\F_2[u]/(u^2)$-code of length $n$. Then, the following identity holds:
\[
\theta_{\Lambda_C}(z) = W_C\left(\theta_0(z), \theta_1(z), \theta_u(z)\right).
\]
\end{theorem}

Before proceeding to the proof of Theorem \ref{zeta8thm2}, we establish the following lemmas.

\begin{lemma}\label{zeta8lem1}
The following identity holds:
\[
\mathrm{Ker}(\rho^{\oplus 1}) = \Lambda.
\]
\end{lemma}

\begin{proof}
From the relations between the bases $f_i$ and $f_i^*$, we have
\begin{align*}
f_1 &= 2f_1^* - f_2^* + f_4^*, \quad f_2 = -f_1^* + 2f_2^* - f_3^*, \\
f_3 &= -f_2^* + 2f_3^* - f_4^*, \quad f_4 = f_1^* - f_3^* + 2f_4^*.
\end{align*} 
Since $f_i \in \mathrm{Ker}(\rho^{\oplus 1})$ for all $1 \leq i \leq 4$, we clearly see that $\Lambda \subset \mathrm{Ker}(\rho^{\oplus 1})$.

Conversely, let $\sum_{i=1}^4 a_i f_i^* \in \mathrm{Ker}(\rho^{\oplus 1})$. By the definition of the kernel, we have $a_1 + a_3 \in 2\Z$ and $a_2 + a_4 \in 2\Z$, which also implies $-a_1 + a_3 \in 2\Z$ and $a_2 - a_4 \in 2\Z$. Hence, we can express the components as
\begin{align*}
a_1 f_1^* + a_3 f_3^* &= a_1 f_1 + \frac{1}{2}(a_1 + a_3)f_2 + a_3 f_3 + \frac{1}{2}(-a_1 + a_3)f_4 \in \Lambda, \\
a_2 f_2^* + a_4 f_4^* &= \frac{1}{2}(a_2 - a_4)f_1 + a_2 f_2 + \frac{1}{2}(a_2 + a_4)f_3 + a_4 f_4 \in \Lambda.
\end{align*}
Thus, $\sum_{i=1}^4 a_i f_i^* \in \Lambda$, which completes the proof.
\end{proof}

\begin{lemma}\label{zeta8lem2}
Let $c = (c_i)_{1 \leq i \leq n} \in (\F_2[u]/(u^2))^{\oplus n}$. Then, we have
\[
(\rho^{\oplus n})^{-1}(c) = \bigoplus_{i=1}^{n} (x_{c_i} + \Lambda),
\]
where $x_0 = 0$, $x_1 = f_1^*$, $x_{1+u} = f_2^*$, and $x_u = f_1^* + f_2^*$.
\end{lemma}

\begin{proof}
Lemma \ref{zeta8lem1} shows that the preimage of each component $c_i$ is given by
\begin{align*}
(\rho^{\oplus 1})^{-1}(c_i) = x_{c_i} + \mathrm{Ker}(\rho^{\oplus 1}) = x_{c_i} + \Lambda.
\end{align*}
Taking the direct sum over all $1 \leq i \leq n$, we obtain the desired assertion.
\end{proof}

\begin{proof}[Proof of Theorem \ref{zeta8thm2}]
Let $c = (c_i)_{1 \leq i \leq n}$ be a codeword in $C$. For the product of the corresponding component theta functions, we have
\begin{align*}
\theta_{c_1}(z) \theta_{c_2}(z)
&= \left( \sum_{x_1 \in x_{c_1} + \Lambda} \exp\left(\pi \sqrt{-1} \Tr_{K/\Q}(\langle x_1, x_1 \rangle_K z)\right) \right) \left( \sum_{x_2 \in x_{c_2} + \Lambda} \exp\left(\pi \sqrt{-1} \Tr_{K/\Q}(\langle x_2, x_2 \rangle_K z)\right) \right) \\
&= \sum_{x_1 \in x_{c_1} + \Lambda} \sum_{x_2 \in x_{c_2} + \Lambda} \exp\left(\pi \sqrt{-1} \Tr_{K/\Q}\left( (\langle x_1, x_1 \rangle_K + \langle x_2, x_2 \rangle_K) z \right)\right) \\
&= \sum_{(x_1, x_2) \in (x_{c_1} + \Lambda) \oplus (x_{c_2} + \Lambda)} \exp\left(\pi \sqrt{-1} \Tr_{K/\Q}\left( \langle (x_1, x_2), (x_1, x_2) \rangle_K z \right)\right).
\end{align*}
Iterating this process for $n$ variables, we obtain
\begin{align*}
\theta_{c_1}(z) \cdots \theta_{c_n}(z) = \sum_{x \in (\rho^{\oplus n})^{-1}(c)} \exp\left(\pi \sqrt{-1} \Tr_{K/\Q}(\langle x, x \rangle_K z)\right),
\end{align*}
where $x = (x_i)_{1 \leq i \leq n} \in (\Lambda^*)^{\oplus n}$ and $\langle x, x \rangle_K = \sum_{i=1}^n \langle x_i, x_i \rangle_K$.

Furthermore, since $\theta_1 = \theta_{1+u}$, the product of these functions can be counted via the Lee composition of $c$:
\[
\sum_{x \in (\rho^{\oplus n})^{-1}(c)} \exp\left(\pi \sqrt{-1} \Tr_{K/\Q}(\langle x, x \rangle_K z)\right) = \theta_0^{N_0(c)}(z) \theta_1^{N_1(c)}(z) \theta_u^{N_2(c)}(z).
\]
Summing this identity over all codewords $c \in C$ yields the statement of Theorem \ref{zeta8thm2}.
\end{proof}

\subsection{Lattice of type $E_6$ over integers of $\Q(\zeta_9)$}\label{zeta9sub}
In this section, similarly to the case of $D_4$, we confirm that a specific lattice of type $E_6$ can be viewed as a $\Z_K$-lattice, and we establish the connections between the weight enumerators of $\Z/3\Z$-codes and the theta functions of lattices over $\Q(\zeta_9)$.


Let $F := \Q(\zeta_9)$ and define the subset $\Lambda \subset F$ by
\[
\Lambda := \frac{1}{3}(1-\zeta_9)(1-\zeta_9^{-1})\Z_F.
\]
Then, $(\Lambda,\Tr_{F/\Q})$ and its dual lattice $(\Lambda^*,\Tr_{F/\Q})$ have $\Z$-bases given by
\[
f_i = \frac{1}{3}(1-\zeta_9)(1-\zeta_9^2)\zeta_9^{i-1} \quad (1\leq i \leq 6),
\]
and
\[
f_i^* =
\begin{cases}
	\frac{1}{3}(1-\zeta_9^4)\zeta_9^{i-3} & \text{if } 1\leq i \leq 3, \\
	\frac{1}{3}(1-\zeta_9^4)(\zeta_9^{i-3}+\zeta_9^{i-6}) & \text{if } 4\leq i \leq 6,
\end{cases}
\] 
respectively.

From the direct calculation of the trace form, we obtain
\[
\left( \Tr_{F/\Q}(f^*_i \overline{f^*_j}) \right)_{1\leq i,j\leq 6} =
\frac{1}{3}
\begin{pmatrix}
    4 &  1 & 1 & 2 & -1 & -1 \\
    1 &  4 & 1 & 2 &  2 & -1 \\
    1 &  1 & 4 & 2 &  2 &  2 \\
    2 &  2 & 2 & 4 &  1 &  1 \\
   -1 &  2 & 2 & 1 &  4 &  1 \\ 
   -1 & -1 & 2 & 1 &  1 &  4
\end{pmatrix}.
\]
By comparing this matrix with the results in Section \ref{E6sub}, we see that $\left(\Lambda, \Tr_{F/\Q}(x\overline{y})\right)$ indeed forms a lattice of type $E_6$.

Let $C$ be a $\Z/3\Z$-code of length $n$. As in Ssection \ref{E6sub}, we apply the additive homomorphism $\rho^{\oplus n} : (\Lambda^*)^{\oplus n} \longrightarrow (\Z/3\Z)^{\oplus n}$ to construct the lifted lattice $\Lambda_C := (\rho^{\oplus n})^{-1}(C) \subset (\Lambda^*)^{\oplus n}$ from the code $C$.

We consider $\Z_K$-lattices in $F$ and $F^{\oplus n}$ analogously to Section \ref{zeta8sub}. Let $K := \Q(\zeta_9+\zeta_9^{-1})$, $V := \Q(\zeta_9)^{\oplus n}$, and let $\Gal(K/\Q) = \{\sigma_1, \sigma_2, \sigma_3\}$. We define a totally positive definite scalar product on $F^{\oplus m}$ by
\[
\langle x, y \rangle_K := \Tr_{F/K}(x\cdot \overline{y}),
\]
where $m$ is a positive integer. Then, $\left(\Lambda, \langle \cdot,\cdot \rangle_K\right)$ and $\left(\Lambda_{C}, \langle \cdot, \cdot \rangle_K\right)$ can be naturally viewed as $\Z_K$-lattices.

From Proposition \ref{K-latthe}, we immediately obtain the following corollary.

\begin{corollary}\label{zeta9cor}
Let $C$ be a Euclidean self-dual code over $\Z/3\Z$ of length $n$. Then, the theta function $\theta_{\Lambda_C}$ of the $\Z_K$-lattice $\Lambda_C$ is a Hilbert modular form of weight $n$ for the whole group $\mathrm{SL}_2(\Z_K)$.
\end{corollary}
Let $z = (z_1, z_2, z_3) \in \HH^{\oplus 3}$. For each $a = 0, 1, 2 \in \Z/3\Z$, we define the theta function $\theta_a(z)$ by
\[
\theta_a(z) := \sum_{x \in x_a + \Lambda} \exp\left(\pi \sqrt{-1} \Tr_{K/\Q}(\langle x,x \rangle_K z)\right),
\]
where the coset representatives are chosen as $x_0 = 0$, $x_1 = f_1^*$, and $x_2 = -f_1^*$. 

Since $-f_1^* + \Lambda = -(f_1^* + \Lambda)$, we have the equality of the sets of norms:
\begin{align*}
\{\langle x,x \rangle_K \mid x \in x_1 + \Lambda\} = \{\langle x,x \rangle_K \mid x \in x_2 + \Lambda\}.
\end{align*}
Hence, it immediately follows that $\theta_1 = \theta_2$.

The weight enumerator of a code $C$ of length $n$ over $\Z/3\Z$ is defined as the polynomial
\[
W_C(X_0, X_1) := \sum_{x \in C} X_0^{\ell_0(x)} X_1^{\ell_1(x)},
\]
where $\ell_0(x)$ denotes the number of zeros in $x$, and $\ell_1(x)$ denotes the number of non-zero elements ($1$ or $2$) occurring in the codeword $x$.

By an argument similar to the proof of Theorem \ref{zeta8thm2}, we obtain the following identity connecting the code and the lattice.

\begin{theorem}\label{zeta9thm2}
Let $C$ be a $\Z/3\Z$-code of length $n$ with $C \subset C^{\perp}$. Then, the following identity holds:
\[
\theta_{\Lambda_C}(z) = W_C\left(\theta_0(z), \theta_1(z)\right).
\]
\end{theorem}

\begin{proof}
It suffices to show that $\mathrm{Ker}(\rho^{\oplus 1}) = \Lambda$. From the relations between the bases $f_i$ and $f_i^*$, we see that
\begin{align*} 
f_1 &= 2f_1^* - f_3^* - f_4^* + f_5^* + f_6^*, \\
f_2 &= 2f_2^* - f_4^* - f_5^* + f_6^*, \\
f_3 &= -f_1^* + 2f_3^* - f_5^* - f_6^*, \\
f_4 &= -f^*_1 - f^*_2 + 2f^*_4 - f^*_6, \\
f_5 &= f^*_1 - f^*_2 - f^*_3 + 2f^*_5, \\
f_6 &= f^*_1 + f^*_2 - f^*_3 - f^*_4 + 2f^*_6.
\end{align*}
Since $f_i \in \mathrm{Ker}(\rho^{\oplus 1})$ for all $1 \leq i \leq 6$, we clearly have $\Lambda \subset \mathrm{Ker}(\rho^{\oplus 1})$.

Conversely, let $\sum_{i=1}^{6} a_i f^*_i \in \mathrm{Ker}(\rho^{\oplus 1})$. By the definition of the map, this implies that $a_1 + a_2 + a_3 - a_4 - a_5 - a_6 \in 3\Z$. Let $S := a_1 + a_2 + a_3 - a_4 - a_5 - a_6$. Then, we can express the element as a linear combination of the basis vectors $f_i$:
\begin{align*}
\sum_{i=1}^{6} a_i f^*_i
&= \frac{1}{3}\left(3(a_1 + a_4) + S\right)f_1 \\
&+ \frac{1}{3}\left(3(a_2 + a_4 + a_5) + S\right)f_2 \\
&+ \frac{1}{3}\left(3(a_3 + a_4 + a_5 + a_6) + S\right)f_3 \\
&+ \frac{1}{3}\left(3(a_1 + a_2 + a_3 + a_4) - S\right)f_4 \\
&+ \frac{1}{3}\left(3(a_2 + a_3 + a_5) - S\right)f_5 \\
&+ \frac{1}{3}\left(3(a_3 + a_6) - S\right)f_6.
\end{align*}
Since $S \in 3\Z$, all coefficients belong to $\Z$. Thus, $\sum_{i=1}^{6} a_i f^*_i \in \Lambda$, which completes the proof.
\end{proof}

\begin{remark}\label{zeta9rem}
Although Theorem \ref{zeta9thm2} is closely related to the case where $p=3$ in \cite[Theorem~5.3]{Ebeling}, they are not identical. Indeed, while the theta function in our theorem involves three variables, the one in \cite{Ebeling} is a single-variable function.
\end{remark}

\end{document}